\documentclass[10pt,journal,twocolumn,twoside]{IEEEtran}
\usepackage{hyperref,cite}
\usepackage{amsmath, amssymb, amsthm, amsfonts}
\usepackage{graphicx}
\usepackage{bm,dsfont,color}
\usepackage{xcolor}
\definecolor{green}{HTML}{006400}
\usepackage{enumitem}
\usepackage{algorithm}
\DeclareMathAlphabet{\mathbbold}{U}{bbold}{m}{n}
\newlist{mylist}{enumerate}{5}
\setlist[mylist]{label=\arabic*.}

\newlist{mylist2}{enumerate}{5}
\setlist[mylist2]{label=\alph*.}
\newlist{mylist3}{enumerate}{5}
\setlist[mylist3]{label=\roman*.}
\def\E{{\mathbb E}}
\def\T{{\mathsf T}}
\def\w{{\boldsymbol w}}

\def\H{{\boldsymbol H}}

\def\v{{\boldsymbol v}}

\def\d{{\boldsymbol d}}
\def\h{{\boldsymbol h}}
\def\bphi{{\boldsymbol \phi}}
\def\bpsi{{\boldsymbol \psi}}
\def\bzeta{{\boldsymbol \zeta}}
\renewcommand{\qed}{\hfill\IEEEQED}

\newtheorem{theorem}{Theorem}
\newtheorem{fact}{Fact}

\newtheorem{assumption}{Assumption}

\newtheorem{remark}{Remark}

\begin{document}
{
\title{Adaptive Penalty-Based Distributed \\Stochastic Convex Optimization}}
%
%
%

\author{Zaid~J.~Towfic,~\IEEEmembership{Student~Member,~IEEE}
        and~Ali~H.~Sayed,~\IEEEmembership{Fellow,~IEEE}
\thanks{The authors are with Department of Electrical Engineering,
University of California, Los Angeles, CA 90095. Email: \{ztowfic, sayed\}@ucla.edu. A short version of this work appears in the conference presentation \cite{towfic2013adaptive}.
This work was supported in part by NSF grant CCF-1011918.
}%
}
\maketitle

\begin{abstract}
In this work, we study the task of distributed optimization over a network of learners in which each learner possesses a convex cost function, a set of affine equality constraints, and a set of convex inequality constraints. We propose a fully-distributed adaptive diffusion algorithm based on penalty methods that allows the network to cooperatively optimize the global cost function, which is defined as the sum of the individual costs over the network, subject to all constraints. We show that when small constant step-sizes are employed, the expected distance between the optimal solution vector and that obtained at each node in the network can be made arbitrarily  small. Two distinguishing features of the proposed solution relative to other related approaches is that the developed strategy does not require the use of projections and is able to adapt to and track drifts in the location of the minimizer due to changes in the constraints or in the aggregate cost itself. The proposed strategy is also able to cope with changing network topology, is robust to network disruptions, and does not require global information or rely on central processors. 
\end{abstract}


\ifCLASSOPTIONpeerreview
\begin{center} \bfseries EDICS Categories: 60. NET-DISP, 77. OPT-DOPT, 63. NET-ADEG\end{center}
\fi
%
\IEEEpeerreviewmaketitle

\allowdisplaybreaks

\section{Introduction}
\label{sec:introduction}
%
%
%
%
\IEEEPARstart{D}{istributed} convex optimization refers to the task of minimizing the aggregate sum of convex cost functions, each available at an agent of a connected network, subject to convex constraints that are also distributed across the agents. The key challenge in such problems is that each agent is only aware of its cost function and its constraints. This article proposes a {\em fully} decentralized solution that is able to minimize the aggregate cost function while satisfying all distributed constraints. The solution method is based solely on local cooperation among neighboring nodes and does not rely on the use of projection constructions. Furthermore, the individual nodes do not need to know any of the constrains besides their own.

There have been several useful studies on distributed convex optimization and estimation techniques in the literature  \cite{Barbarossa,PaoloAlgebraic,bertrandSPM,lee2013distributed,
                      tsitsiklis1986distributed,bertsekas1989parallel,kar2008distributed,theodoridis2011adaptive,Takahashi2008139,Djuric,cui2007estimation,hu2011adaptive,
                      ChouvardasAdaptiveRobust2011,dini2012cooperative,Li20082599,Krishnamurthy,boyd2006randomized,
                      boyd2011distributed,palomar2010convex,yan,Nedic}. Most existing techniques are suitable for the solution of {\em static} optimization problems, where the objective is to determine the location of a fixed optimal parameter. The available solution methods tend to employ constructions that become problematic in the context of adaptation and learning over networks. This is because they rely on the use of decaying step-sizes in their stochastic gradient updates \cite{Nedic,yan,lee2013distributed,tsitsiklis1986distributed}. And it is well-known that decaying step-sizes are a hindrance to adaptation when it is desired to develop {\em dynamic} or adaptive solutions that are able to track {\em drifts} in the location of the optimal parameter; these drifts can result from changes in the constraint conditions or in the cost functions themselves. For this reason, in this work, we employ {\em constant} step-sizes in order to enable continuous adaptation and learning.
                      
	When constant step-sizes are used, the dynamics of the distributed algorithm is changed in a nontrivial manner and its convergence analysis becomes more demanding because, as we are going to see, the gradient update term does not die out anymore  with time as happens with decaying step-size implementations. In the constant step-size case, gradient noise will always be present and will seep into the update equations. Nevertheless, we will be able to show that the proposed distributed strategy  can still ensure approximation errors of the order of the step-size so that arbitrarily small levels of accuracy can be attained by using sufficiently small step-sizes (see Theorem \ref{thm:alg_error}).

We further note that most available distributed solutions rely on the use of projection steps in order to ensure that the successive estimates at the nodes satisfy the convex constraints \cite{lee2013distributed,SrivastavaDistributed2011,palomar2010convex,yan,Nedic}. In some of the methods \cite{Nedic,yan}, each node is required to know all the constraints across the entire network in order to compute the necessary projections. Clearly, this requirement defeats the purpose of a distributed solution since it requires the nodes to have access to global information. The works \cite{lee2013distributed,SrivastavaDistributed2011} develop useful distributed solutions where nodes are only required to know their own constraints. However, the constraint conditions still need to be relatively simple in order for the distributed algorithm to be able to compute the necessary projections analytically (such as projecting  onto the nonnegative orthant).  In cases when the constraints are more complex so that the necessary projections are not easily computed, then several of the existing techniques tend to implement an {\em offline} optimization routine that is guaranteed to converge only asymptotically, and not in a finite number of steps, as explained in  \cite{cvx_book,Polyak}. The analysis for these methods generally assumes that the projection step is implemented ideally even though the offline iterations are in fact truncated in practice and the truncation errors interfere with the accuracy of the distributed solution. 

Motivated by the above considerations, in this work, we propose a distributed solution that employs constant step-sizes and that eliminates the need for projection steps. The solution relies instead on the use of suitably chosen {\em penalty} functions and replaces the projection step by a stochastic approximation update that is made to run {\em simultaneously} with the optimization step. The challenge is to show that the use of  penalty functions in the stochastic gradient update step still leads to accurate solutions. The analysis in the article establishes that this is indeed possible. In particular, we show following Theorem \ref{thm:alg_error} further ahead how to select the parameters of the proposed algorithm in order to ensure desirable convergence properties with small approximation errors. Moreover, in the proposed solution, the nodes are only required to interact locally and to have access to local estimates from their neighbors; there is no need for the nodes to know any of the constrains besides their own.

One important issue that is useful to mention is that some solution methods (e.g., \cite{Nedic}) require a {\em feasible} initial condition for their distributed algorithm. When the constraint set is distributed across the agents, it is not possible to find such feasible initial conditions without a substantial amount of in-network communication. We therefore take a different approach. By relying on suitably defined stochastic approximation steps, we show how the weight estimates constructed by the various nodes will approach the optimal feasible solution with arbitrarily good precision.

	  The technique used in this work relies on the use of diffusion strategies, which have been proven to have useful convergence and learning properties \cite{Neurocomputing,Jianshu_diff_wo,shine_diffusion_consensus,sayedMagazine}. The algorithm is comprised of three steps: 1) an adaptation step that updates the current solution using the local stochastic gradient available at the current iteration; 2) a constraint penalty step that penalizes directions that are not feasible according to the local constraint set; 3) and an aggregation step in which each agent combines its solution estimate with that of its \emph{network neighbors}. In this way, the only communication that takes place in the algorithm is in-network and relatively low-power since neighbors are usually (but not necessarily) chosen according to physical proximity. 
	  
	  \noindent {\bf Notation}. Throughout the manuscript, random quantities are denoted in boldface. Matrices are denoted in capital letters while vectors and scalars are denoted in small-case letters. The operator $\preceq$  denotes an element-wise inequality; i.e., $a \preceq b$ implies that each pair of elements of the vectors $a$ and $b$ satisfy $a_i \leq b_i$.

\section{Background: Augmentation Methods}
\label{sec:background_augmentation}
In this section, we briefly review a basic technique in constrained deterministic optimization and highlight some of the issues that are relevant to distributed implementations and that need attention. Specifically, we describe  augmentation-based methods for constrained optimization. These methods generally fall into two categories: (1) barrier methods, also known as interior penalty methods, and (2) penalty methods, also known as exterior penalty methods. Both methods are based on a simple yet insightful technique to augment the original objective function with a ``penalty'' term that penalizes getting too close to the constraint from the interior of the feasible set or leaving the feasible region altogether. 

Thus, consider a convex optimization problem of the form:
\begin{alignat}{3}
	&\min_w\  &&J(w) \label{eq:original_objective}\\
	&\;\textrm{subject\ to\ }\  &&g_{l}(w) \leq 0,\quad\quad l=1,2,\ldots,L \nonumber
\end{alignat}
where $w\in\mathbb{R}^M$, $\{g_1(w),\ldots,g_L(w)\}$ is a collection of convex functions, and $J(w)$ is a strongly convex function from $\mathbb{R}^M$ to $\mathbb{R}$. Augmentation incorporates the inequality constraints into the cost function and helps transform the constrained optimization problem into an unconstrained optimization problem via a \emph{convex} barrier or penalty function $\delta(\cdot): \mathbb{R}\rightarrow\mathbb{R}$, in the following manner:
\begin{align}
	\min_w\  J(w) + \eta \sum_{l=1}^L \delta(g_l(w))
	\label{eq:aumented_cost}
\end{align}
where $\eta > 0$ is a scalar parameter that controls the relative importance of adhering to the constraints. One choice for $\delta(\cdot)$ that yields an equivalent problem to \eqref{eq:original_objective} for any finite $\eta > 0$ is the indicator function \cite[pp.~562--563]{cvx_book}:
	\begin{align}
		\delta^\textrm{IF}(x) = \begin{cases}
										\ \ \!0, & x \leq 0\\
										\infty, & \textrm{otherwise}
								   \end{cases}
								   \label{eq:indicator_function}
	\end{align}
Observe that the indicator function $\delta^\textrm{IF}(x)$ is convex and nondecreasing. Since the indicator function is generally nondifferentiable, approximations are used in its place. The main difference between barrier methods and penalty methods is the choice of the approximating functions.
\subsection{Barrier Method}
\label{ssec:barrier}
	Barrier methods set a ``barrier'' around the feasible region. One of the most popular smooth approximations for \eqref{eq:indicator_function} is the logarithmic barrier function:
	\begin{align}
		\delta^{\log}(x) = \begin{cases}
							-\log(-x), & x < 0\\
							\quad\quad\quad\ \infty, &\textrm{otherwise}
						  \end{cases} \label{eq:log_barrier}
	\end{align}
	In this case, the algorithm requires a strictly feasible initialization, so that the augmented cost given in \eqref{eq:aumented_cost} is finite. A gradient-descent optimization algorithm would then travel against the gradient of \eqref{eq:aumented_cost}, while adjusting the step-size to ensure that the next iterate stays within the feasible region via a line-search algorithm \cite[p.~464]{cvx_book} \cite[p.~288]{Polyak}. Barrier methods are \emph{interior-point methods} since the iterates never leave the feasible-set. Clearly, this is an advantage since any solution obtained during the optimization process may be used as a sub-optimal approximation. Nevertheless, this advantage requires a strictly feasible initialization. When the entire constraint set $\{g_1(w),\ldots,g_L(w)\}$ is not available to an agent (as happens in {\em distributed} constrained optimization), then it is not possible to choose a strictly feasible initializer without sharing this global information with the agents. This situation creates an annoying disadvantage from the perspective of distributed optimization. We will see that penalty methods avoid this difficulty.

\subsection{Penalty Method}
\label{sec:penalty_method}
	In contrast to barrier methods, penalty methods give some positive penalty to solutions that fall outside the feasible set. In this case, the inequality penalty function takes the form:
	\begin{align}
		\delta^\textrm{IP}(x) &= \begin{cases}
						0, &x \leq 0\\
						>0, &\textrm{otherwise}
					 \end{cases}
					 \label{eq:inequality_penalty_condition}
	\end{align}
	One continuous, convex, and twice-differentiable choice that satisfies \eqref{eq:inequality_penalty_condition} is:
	\begin{align}
		\delta^\textrm{SIP}(x) &= \max(0,x^3) \label{eq:SIP} 
	\end{align}
	Observe that $\delta^\textrm{SIP}(x)$ does not assume unbounded values for bounded $x$ and, therefore, penalty methods do not require a feasible solution as an initializer. While this fact implies that penalty methods are particularly well-suited for distributed optimization scenarios, it also follows that the iterates may not remain inside the feasible region in general. This property means that there is no longer a need to execute a linesearch backtracking algorithm in gradient-descent implementations and, therefore, the step-sizes may assume {\em constant} values throughout the execution of the algorithm. The use of constant step-sizes is advantageous for a couple of reasons. First, it allows us to reduce the number of free parameters in the algorithm. Second, it becomes possible to derive useful bounds on the performance of the algorithm. And, perhaps more importantly, constant step-sizes endow the resulting distributed algorithm with adaptation and learning abilities. In this way, the algorithm acquires the ability to track in real-time variations in the underlying constraints and in the location of the minimizer. In comparison, diminishing step-sizes are problematic because once these step-sizes approach their zero limiting value, the algorithm stops adapting.
	
	For penalty methods, we observe that the approximation \eqref{eq:inequality_penalty_condition} of \eqref{eq:indicator_function} improves in quality as $\eta$ increases in value \cite[p.~366]{Bazaraa} \cite[p.~288]{Polyak}. This is because the penalty on the inside of the feasible region is zero and does not increase as $\eta$ is increased. However, as $\eta\rightarrow\infty$, the function $\eta\cdot \delta^\textrm{IP}(x)$ approximates the ideal barrier \eqref{eq:indicator_function}. At that stage, expression \eqref{eq:aumented_cost} would have the shape of the original cost function over the feasible set, and the effective objective would be infinite outside the feasible set. Since $J(w)$ and the penalty function, $\delta^\textrm{IP}(g_l(w))$, are convex, the augmented cost is also convex and its minimizer is obtained at the optimizer of the original optimization problem as  $\eta\rightarrow\infty$ \cite[p.~366]{Bazaraa}.
	
	Another advantage of penalty methods, as opposed to barrier methods, is that it is possible to easily incorporate {\em affine} constraints as well. Thus, consider the convex optimization problem:
	\begin{alignat}{3}
		&\min_w\ \ &&J(w) \label{eq:original_equality}\\
		&\textrm{subject\ to\ }\  &&h_{u}(w) = 0,\quad\quad u=1,2,\ldots,U \nonumber\\
		&					    &&g_{l}(w) \leq 0,\quad\quad\ l\,=1,2,\ldots,L \nonumber
	\end{alignat}
	where the functions $h_u(w)$ are affine. This problem can also be approached as an unconstrained optimization problem using penalty functions:
	\begin{align}
	\min_w\ \!J(w) + \eta\left[\sum_{l=1}^L \delta^\textrm{IP}(g_l(w)) + \sum_{u=1}^U \delta^\textrm{EP}(h_u(w))\right]
	\label{eq:augmented_cost_equality}
	\end{align}
	where $\delta^\textrm{IP}(\cdot)$ is described in \eqref{eq:inequality_penalty_condition} while $\delta^\textrm{EP}(\cdot): \mathbb{R}\rightarrow\mathbb{R}$ is a convex function that is described by 
	\begin{align}	
		\delta^\textrm{EP}(x) = \begin{cases}
			0, & x=0\\
			>0, & x\neq 0
		\end{cases}
		\label{eq:equality_penalty_condition}
	\end{align}
	One popular choice of a continuous, convex, and twice-differentiable equality penalty function that satisfies \eqref{eq:equality_penalty_condition} is the quadratic penalty:
	\begin{align}
		\delta^\textrm{SEP}(x) &= x^2 \label{eq:SEP}
	\end{align}
	Clearly, since the penalty functions are convex and the original objective function is strongly convex, the augmented cost \eqref{eq:augmented_cost_equality} remains strongly convex. Moreover, when \eqref{eq:original_equality} is feasible, the minimizer of \eqref{eq:augmented_cost_equality} tends to the optimal solution of the original problem \eqref{eq:original_equality} as $\eta\rightarrow \infty$ (see Theorem \ref{thm:approx_error}). This shows that it is possible to tackle both equality and inequality constraints simultaneously using penalty methods. Table \ref{tbl:Barrier_Penalty} lists the advantages and disadvantages of the barrier and penalty methods for the distributed optimization problem under study.
	
\begin{table*}
\caption{Table listing the advantages and disadvantages of the barrier and penalty methods for distributed constrained optimization}
\label{tbl:Barrier_Penalty}
\centering
{\begin{tabular}{c|c|c|c|c|c}
	\hline 
	\textbf{Method} & \textbf{Feasible Start}  & \textbf{Incorporate Equality Constraints} & \textbf{Full Knowledge of Feasible Set} & \textbf{Iterates Feasible} & \textbf{Constant Step-size}\\ 
	\hline \hline
	\textbf{Barrier} & Required & Indirectly & Required & Guaranteed & No (Backtracking)\\ 
	\hline 
	\textbf{Penalty} & Not Required & Directly & Not Required & Asymptotically & Yes\\ 
	\hline 
	\end{tabular} }	
\end{table*}

	In the next section, we will examine how penalty methods can be effectively used in distributed convex optimization algorithms to obtain the solution of the original optimization problem \eqref{eq:original_equality} without explicitly communicating the constraints across the agents in the network.
	
\section{Constrained Optimization over Networks}
Consider a network of agents (nodes), where each node $k$ possesses a strongly convex cost function, $J_k(w)$, and a convex set of constraints $w\in\mathbb{W}_k$ where $w\in\mathbb{R}^M$. The objective of the network is to optimize the aggregate cost across all nodes subject to all constraints, i.e.,
\begin{alignat}{3}
	&\!\min_w\ &&J^\textrm{glob}(w) \triangleq \sum_{k=1}^N J_k(w) \label{eq:orig_dist_problem}\\
	&\textrm{subject\ to\ \ }&& w \in \mathbb{W}_1, \ldots,w\in\mathbb{W}_N\nonumber
\end{alignat}
Each of the convex sets $\{\mathbb{W}_1,\ldots,\mathbb{W}_N\}$ is defined as the set of points $w$ that satisfy a collection of affine equality and convex inequality constraints:
\begin{align}
	\mathbb{W}_k \triangleq \left\{w: \begin{aligned}
								h_{k,u}(w) = 0, & \quad\quad u=1,\ldots,U_k\\
								g_{k,l}(w) \leq 0, & \quad\quad l=1,\ldots,L_k
							\end{aligned} \right.
\end{align}
Obviously, the original optimization problem \eqref{eq:orig_dist_problem} can be cast as the optimization of the aggregate cost function $J^\textrm{glob}(w)$ over the common feasible set, $\mathbb{W}_1 \cap \ldots \cap \mathbb{W}_N$:
\begin{alignat}{3}
	&\!\min_w\ && J^\textrm{glob}(w)\quad \textrm{subject\ to\ \ } w \in \mathbb{W}
\end{alignat}
where $\mathbb{W} \triangleq \mathbb{W}_1 \cap \ldots \cap \mathbb{W}_N$ is a convex set since the intersection of convex sets is itself convex \cite[p.~36]{cvx_book}. Assuming a solution for the above deterministic optimization problem exists (i.e, $\mathbb{W} \neq \emptyset$), we will denote an optimal solution for it by $w^\star$. The optimal objective value is given by $J^\textrm{glob}(w^\star)$. Observe that since $J^\textrm{glob}(w)$ is strongly-convex, then $w^\star$ is unique  (see Fact \ref{fact:uniqueSolution} further ahead).\\

\noindent \begin{remark} {\em Although we are requiring the individual cost functions $J_k(w)$ to be strongly convex, this condition is actually unnecessary and it is sufficient to require that at least one of the individual costs is strongly convex while all other costs can simply be convex; this condition is sufficient to ensure that the aggregate cost $J^{\rm glob}(w)$ will remain strongly convex. Most of the results in this manuscript, and especially the convergence results and the conclusions of Facts \ref{fact:uniqueSolution} and \ref{fact:approx_unique} and Theorems \ref{thm:approx_error}-\ref{thm:alg_error}, will hold under these weaker conditions --- see the explanation given in Remark \ref{rem:remark2} in Appendix \ref{app:alg_error} following equation \eqref{eq:bound_stability2}. The strong convexity of the individual costs is adopted here for three reasons. First, the more relaxed situation would require more technical  arguments to arrive at the same conclusions, as shown in \cite{chen2012limiting} in a different context. Due to space limitations, we opt to illustrate our construction under the strong convexity condition to facilitate the exposition of the main conclusions without digressing into specialized situations. Second, strong convexity is satisfied in many applications involving adaptation and learning where it is common to incorporate regularization into the cost functions. Regularization automatically ensures strong convexity. Third, when strong convexity is not satisfied, the Hessian matrices of the individual costs can become close-to-singular and ill-conditioned, which is known to be problematic for real-time implementations using streaming data.} \hfill\qed
\end{remark}

Returning to \eqref{eq:orig_dist_problem}, using the cost-augmentation technique described in Sec.~\ref{sec:background_augmentation}, we approximate \eqref{eq:orig_dist_problem} by using penalty functions in a manner similar to \eqref{eq:augmented_cost_equality}. Specifically, we consider the unconstrained problem:
\begin{align}
\min_w\ J^\textrm{glob}_{\eta}(w) \label{eq:global_cost_approx}
\end{align}
where
\begin{align}
J^\textrm{glob}_{\eta}(w)&\triangleq \sum_{k=1}^N J_k(w) + \eta \sum_{k=1}^N p_k(w)\label{eq:J^approx}
\end{align}
and 
\begin{align}
p_k(w) \triangleq \sum_{l=1}^{L_k} \!\delta^\textrm{IP}\!(g_{k,l}(w)) \!+ \sum_{u=1}^{U_k} \!\delta^\textrm{EP}\!(h_{k,u}(w)) \label{eq:p}
\end{align}
with $\delta^\textrm{IP}(x)$ and $\delta^\textrm{EP}(x)$ denoting continuous convex functions that satisfy \eqref{eq:inequality_penalty_condition} and \eqref{eq:equality_penalty_condition}, respectively. We assume that $\delta^\textrm{IP}(x)$ and $\delta^\textrm{EP}(x)$ are selected so that $\nabla_w p_k(w') = 0$ when $w' \in \mathbb{W}$ (this is the case, for example, for \eqref{eq:SIP} and \eqref{eq:SEP}). We stress that \eqref{eq:global_cost_approx} is {\em{not}} an equivalent problem to \eqref{eq:orig_dist_problem} when the indicator function \eqref{eq:indicator_function} is not utilized, but is an approximation for it. We will see later though that the approximation improves as $\eta \rightarrow \infty$. When $J^\textrm{glob}(w)$ is strongly convex, the cost \eqref{eq:global_cost_approx} will also be strongly-convex and will have a unique optimizer for any $\eta > 0$ (see Fact \ref{fact:approx_unique} further ahead). We shall denote this optimal solution to \eqref{eq:global_cost_approx} by $w^o(\eta)$, which is parameterized in terms of $\eta$. Our task is now two-fold: (1) to motivate a fully distributed algorithm to solve \eqref{eq:global_cost_approx} and determine $w^{o}(\eta)$, and (2) to characterize the distance between $w^o(\eta)$ and the desired optimizer $w^{\star}$ of  \eqref{eq:orig_dist_problem}. The distributed solution that we develop will rely solely on local in-network processing with each agent having knowledge of only its own constraint set $\mathbb{W}_k$. We will establish after Theorem \ref{thm:alg_error} in the sequel that by choosing the algorithm's parameters appropriately, it is possible to obtain an arbitrarily accurate approximation for $w^{\star}$. 


\subsection{Diffusion-Based Distributed Optimization}
	Consider the optimization problem given by \eqref{eq:global_cost_approx}. Its aggregate cost can be expressed as the sum of local cost functions as follows:
	\begin{align}
		J^{\textrm{glob}}_{\eta}(w) \triangleq \sum_{k=1}^N J_{k,\eta}'(w)
		\label{eq:J^glob'}
	\end{align}
	where
	\begin{align}
	J_{k,\eta}'(w) \triangleq J_k(w) \!+\! \eta \cdot p_k(w)
	\label{eq:J_k'}
	\end{align}
	and $p_k(w)$ is defined in \eqref{eq:p}. Observe that each function $J_{k,\eta}'(w)$ depends only on agent $k$'s information: cost function $J_k(w)$ and constraint set $\mathbb{W}_k$. This situation falls within the framework of unconstrained diffusion optimization \cite{Jianshu_common_wo,Jianshu_diff_wo}. Following similar arguments to those employed in these references, we conclude that one way to seek the minimizer of \eqref{eq:J^glob'} is for each node to run iterations of the following form with a constant step-size:
	\begin{subequations}
	\begin{align}
		\psi_{k,i} &= w_{k,i-1} - \mu \cdot \nabla_w J_{k,\eta}'(w_{k,i-1}) \label{eq:A_0}\\
		w_{k,i} &= \sum_{\ell=1}^N a_{\ell k} \psi_{\ell,i} \label{eq:C_0}
	\end{align}
	\end{subequations}
	In \eqref{eq:A_0}-\eqref{eq:C_0}, the vector $w_{k,i-1}$ denotes the estimate for $w^o(\eta)$ at node $k$ at iteration $i-1$. This iterate is first updated via the (adaptive) gradient-descent update \eqref{eq:A_0} with step-size $\mu > 0$ to the intermediate value $\psi_{k,i}$. All other nodes in the network perform a similar update simultaneously by using their gradient vectors. Subsequently, each node $k$ uses \eqref{eq:C_0} to combine, in a convex manner, the intermediate estimates from its neighbors. This step results in the updated estimate $w_{k,i}$ and the process repeats itself. The nonnegative coefficients $\{a_{\ell k}\}$ are chosen to satisfy the conditions:
	\begin{subequations}
	\begin{align}
		\!\!a_{\ell k} &= 0,\quad \textrm{when agents\ } \ell \textrm{\ and\ } k \textrm{\ are not neighbors} \label{eq:conditions_a_1}\\
		\!\!\sum_{\ell=1}^N a_{\ell k} &= 1, \quad k = 1,\ldots,N \label{eq:conditions_a_2}
	\end{align}
	\end{subequations}
	If we collect these coefficients into a matrix $A=[a_{\ell k}]$, then condition \eqref{eq:conditions_a_2} implies that $A$ is left-stochastic (i.e., it satisfies $A^\T \mathds{1}_N = \mathds{1}_N$, where $\mathds{1}_N \in \mathbb{R}^N$ is the vector with all entries equal to one). 
	
	Evaluating the gradient vector from \eqref{eq:J_k'} and substituting into \eqref{eq:A_0} we get:
	\begin{align}
		\psi_{k,i} &= w_{k,i-1} - \mu\cdot \nabla_w J_k(w_{k,i-1}) - \mu \eta \cdot \nabla_w p_k(w_{k,i-1}) \label{eq:psi}
	\end{align}
	for differentiable penalty functions. Expression \eqref{eq:psi} indicates that the update from $w_{k,i-1}$ to $\psi_{k,i}$ involves two components: the original gradient vector, $\nabla_w J_k(\cdot)$, and the gradient vector of the penalty function.  We can incorporate these update terms into $w_{k,i-1}$ in various orders. One convenient way to express the update is to split it into two parts: first we move from $w_{k,i-1}$ to $\psi_{k,i}$ in the opposite direction of the gradient vector of $J_k(\cdot)$. Subsequently, we incorporate the correction by the penalty gradients, say, as follows:
	\begin{subequations}
	\begin{align}
		\zeta_{k,i} &= w_{k,i-1} - \mu \cdot \nabla_w J_k(w_{k,i-1}) \label{eq:zeta}\\
		\psi_{k,i} &= \zeta_{k,i} - \mu\eta\cdot \nabla_w p_k(w_{k,i-1}) \label{eq:psi_2}
	\end{align}
	\end{subequations}
	It is generally expected that the intermediate iterate $\zeta_{k,i}$ generated by \eqref{eq:zeta} is a better estimate for $w^o(\eta)$ than $w_{k,i-1}$. This motivates us to replace $w_{k,i-1}$ in \eqref{eq:psi_2} by $\zeta_{k,i}$ to get:
	\begin{align}
		\zeta_{k,i} &= w_{k,i-1} - \mu \cdot \nabla_w J_k(w_{k,i-1}) \label{eq:zeta_orig}\\
		\psi_{k,i} &= \zeta_{k,i} - \mu\eta\cdot \nabla_w p_k(\zeta_{k,i}) \label{eq:psi_orig}
	\end{align} 
	This last substitution is reminiscent of incremental-type arguments in gradient descent algorithms \cite{bertsekas1997new,nedic2001incremental,rabbat2005quantized}. We further observe from \eqref{eq:p} that the gradient vector of the penalty function can in turn be decomposed into the sum of two gradient components: one arising from the inequality constraints and the other from the equality constraints. Thus, in principle, we can further split \eqref{eq:psi_orig} into two steps by adding these two gradient components one at a time. We shall forgo this extension here since \eqref{eq:zeta_orig}--\eqref{eq:psi_orig} is sufficient to convey the idea behind the main construction in this article. Further splitting of the gradient updates can generally help improve the performance of the distributed algorithm; this study can be pursued using techniques similar to those used by \cite{jaewoo}.

	Now, combining \eqref{eq:zeta_orig}--\eqref{eq:psi_orig} with \eqref{eq:C_0}, we arrive at what we shall refer to as the {\em penalized} Adapt-then-Combine (ATC) diffusion algorithm shown in Eqs. \eqref{eq:ATC_A}--\eqref{eq:ATC_C},
	\begin{algorithm}
		\caption{Diffusion Adapt-then-Combine (ATC)} 
		 \begin{subequations}
			\begin{align}
				\zeta_{k,i} &= w_{k,i-1} - \mu \cdot\nabla_w J_k(w_{k,i-1}) \label{eq:ATC_A}\\
				\psi_{k,i} &= \zeta_{k,i} - \mu\eta\cdot \nabla_w p_k(\zeta_{k,i})\label{eq:ATC_P}\\
				w_{k,i} &= \sum_{\ell \in \mathcal{N}_k} a_{\ell k} \psi_{\ell,i} \label{eq:ATC_C}
			\end{align}
		 \end{subequations}
	\end{algorithm}
	where $\mathcal{N}_k$ denotes the neighborhood of node $k$. It is also possible to interchange the order in which steps \eqref{eq:A_0}--\eqref{eq:C_0} are performed, with combination performed prior to adaptation. Following similar arguments to the above, we can motivate the alternative {\em penalized} Combine-then-Adapt (CTA) diffusion algorithm shown in Eqs. \eqref{eq:CTA_C}--\eqref{eq:CTA_P}.
	Observe that in both penalized ATC and CTA algorithms,  there is an explicit step to move along the gradient of the penalty function. This step can be thought of as performing a single incremental ``projection'' step along agent $k$'s constraints \cite[pp.~20-21]{Polyak}. Before we move on to establish the convergence of these distributed strategies for sufficiently small step-sizes, we pause to compare their structure with other related contributions in the literature. 
	\begin{algorithm}
		\caption{Diffusion Combine-then-Adapt (CTA)} 
		 \begin{subequations}
			\begin{align}
				\psi_{k,i-1} &= \sum_{\ell \in \mathcal{N}_k} a_{\ell k} w_{\ell,i-1} \label{eq:CTA_C}\\
				\zeta_{k,i} &= \psi_{k,i-1} - \mu\cdot \nabla_w J_k(\psi_{k,i-1}) \label{eq:CTA_A}\\
				w_{k,i} &= \zeta_{k,i} - \mu\eta\cdot \nabla_w p_k(\zeta_{k,i}) \label{eq:CTA_P}
			\end{align}
		 \end{subequations}
	\end{algorithm}
	\subsection{Comparison with Consensus-Based Constructions}
	We first compare the penalized CTA algorithm \eqref{eq:CTA_C}--\eqref{eq:CTA_P} to the consensus-based algorithm used in \cite{yan} for constrained optimization, and which is reproduced below using our notation: 
	\begin{subequations}
	\begin{align}
		\psi_{k,i-1} &= \sum_{\ell \in \mathcal{N}_k} a_{\ell k} w_{\ell,i-1}\label{eq:consensus_1}\\
		\zeta_{k,i} &= \psi_{k,i-1} - \mu \cdot \nabla_w J_k(w_{k,i-1}) \label{eq:consensus_1_5}\\
		w_{k,i} &= P_{\mathbb{W}_1\cap\ldots\cap\mathbb{W}_N} \left[\zeta_{k,i}\right] \label{eq:consensus_2}
	\end{align}
	\end{subequations}
	where the notation $P_\mathbb{X}[y]$ denotes the operation of projecting the vector $y$ onto the set $\mathbb{X}$:
	\begin{align}
		P_\mathbb{X}[y] \triangleq \underset{x \in \mathbb{X}}{\arg\min}\ \|x-y\| \label{eq:projection}
	\end{align}
	Observe that the gradient vector in \eqref{eq:consensus_1_5} is evaluated at the old iterate, $w_{k,i-1}$, and not at the updated iterate $\psi_{k,i-1}$ as in \eqref{eq:CTA_A}. Moreover, the projection step \eqref{eq:consensus_2} corresponds to  {\em multiple} (in principle, infinite)  iterations of the final step \eqref{eq:CTA_P} of the penalized CTA algorithm and assumes global knowledge of the full feasible set $\mathbb{W}$ by node $k$. This assumption is a hindrance to distributed implementations. Moreover, unless the constraints are simple, the actual projection in \eqref{eq:consensus_2} is usually found via augmentation methods such as the barrier method discussed in Sec.~\ref{ssec:barrier}, and enough iterations need to be executed offline until, for example, the norm of the gradient vector is sufficiently small. We therefore note that the consensus-based implementation \eqref{eq:consensus_1}--\eqref{eq:consensus_2} requires the sharing of global information among all nodes and the algorithm involves two separate time-scales: a slower scale for performing \eqref{eq:consensus_1}--\eqref{eq:consensus_2} and a faster scale for running the multiple iterations that are needed to carry out the projection needed for step \eqref{eq:consensus_2}.

	Furthermore, it has been shown recently in the literature that performing the combination and adaptation steps incrementally, where the updated iterate $\psi_{k,i-1}$ is used in the gradient vector in \eqref{eq:CTA_A}, guarantees network stability in mean-square-error optimization problems  while consensus-based implementations using \eqref{eq:consensus_1_5} can become unstable. The reason is the following. Note that the same weight estimate $\psi_{k,i-1}$ is used on the right-hand side of the diffusion update \eqref{eq:CTA_A}, while different estimates $\{\psi_{k,i-1},w_{k,i-1}\}$ are used on the right-hand side of the consensus update \eqref{eq:consensus_1_5}. This asymmetry can cause an unbounded growth in the state of consensus networks and lead to instability, as explained in \cite{shine_diffusion_consensus}.
	
	For this reason, we shall continue our presentation by focusing on the penalized CTA and ATC diffusion strategies \eqref{eq:ATC_A}--\eqref{eq:ATC_C} and \eqref{eq:CTA_C}--\eqref{eq:CTA_P}.
\subsection{Comparison with Projection-Based Constructions}
Another distributed algorithm is developed in \cite{lee2013distributed}; it relies on a structure similar to the penalized CTA diffusion form albeit with two important differences: step \eqref{eq:CTA_P} is replaced by the local projection step \eqref{eq:Lee_P} shown below and the constant step-size in step \eqref{eq:CTA_A} is replaced by an iteration-dependent step-size in step \eqref{eq:Lee_A}:
\begin{subequations}
	\begin{align}
		\psi_{k,i-1} &= \sum_{\ell \in \mathcal{N}_k} a_{\ell k} w_{\ell,i-1} \label{eq:Lee_C}\\
		\zeta_{k,i} &= \psi_{k,i-1} - \mu(i)\cdot \nabla_w J_k(\psi_{k,i-1}) \label{eq:Lee_A}\\
		w_{k,i} &= P_{\mathbb{W}_k} \left[\zeta_{k,i}\right] \label{eq:Lee_P}
	\end{align}
\end{subequations}
In this solution, each node does not need to know the global constraint set $\mathbb{W}$ and would project only onto agent $k$'s constraint set $\mathbb{W}_k$, as indicated by \eqref{eq:Lee_P}. However, and understandably, each constraint set $\mathbb{W}_k$ is required to consist of ``simple constraints'' whose projections \eqref{eq:Lee_P} can be computed analytically, such as the projection onto the non-negative orthant. As explained earlier, the solution method we propose in this work removes the need for carrying out explicit projection steps such as \eqref{eq:Lee_P}. Moreover, note that step \eqref{eq:Lee_A} utilizes a diminishing step-size, which limits the adaptation ability of the network in tracking drifting constraints and cost functions under {\em dynamic} optimization scenarios. For this reason, we are setting the step-size to a constant value in \eqref{eq:CTA_A}. By doing so, the dynamics of the algorithm changes in a significant manner. For one thing, with a constant step-size, the right-most gradient term in \eqref{eq:Lee_A} would not vanish anymore (because the step-size does not vanish anymore) and the algorithm will continue to adapt indefinitely. It then becomes necessary to examine whether the algorithm would still be able to approach the solution of the optimization problem with high accuracy due to persistent gradient noise. The main results in this paper establish that this is indeed the case for the proposed penalized diffusion implementations.
\section{Analysis Setup and Main Assumptions}
In this section, we study the performance of the penalized algorithms \eqref{eq:ATC_A}-\eqref{eq:ATC_C} and \eqref{eq:CTA_C}-\eqref{eq:CTA_P} in a unified manner. We shall not limit our analysis to deterministic optimization problems, but will consider more general stochastic gradient approximation problems where the true gradient vectors, $\nabla_w J_k(\cdot)$, are replaced by approximations, say, $\widehat{\nabla_w} J_k(\cdot)$. We model the approximate gradient direction as a randomly perturbed version of the true gradient, say, as:
	\begin{align}
		\widehat{\nabla_w} J_k(w) \triangleq \nabla_w J_k(w) + \v_{k,i}(w)
	\end{align}
	where $\v_{k,i}(\cdot)$ is the perturbation vector (or gradient noise). Observe that once we replace $\nabla_w J_k(w)$ by $\widehat{\nabla_w} J_k(w)$, then the variables $\phi$, $\psi$, $\zeta$, and $w$ in the diffusion strategies \eqref{eq:ATC_A}--\eqref{eq:ATC_C} and \eqref{eq:CTA_C}--\eqref{eq:CTA_P} become random variables due to the presence of the random perturbation $\v_{k,i}(\cdot)$. 
	
  In order to treat the two penalized diffusion algorithms (ATC and CTA) within a unified framework, we consider the following general description:
\begin{subequations}
\begin{align}
	\bphi_{k,i-1} &= \sum_{\ell \in \mathcal{N}_k} a_{1,\ell k} \w_{\ell,i-1} \label{eq:diff_C1}\\
		\bzeta_{k,i} &= \bphi_{k,i-1} - \mu \cdot \widehat{\nabla_w} J_k(\bphi_{k,i-1}) \label{eq:diff_A}\\
		\bpsi_{k,i} &= \bzeta_{k,i} - \mu\eta\cdot \nabla_w p_k(\bzeta_{k,i}) \label{eq:diff_P}\\
		\w_{k,i} &= \sum_{\ell \in \mathcal{N}_k} a_{2,\ell k} \bpsi_{\ell,i} \label{eq:diff_C2}
\end{align}
\end{subequations}
where we introduced two sets of nonnegative convex combination coefficients $\{a_{1,\ell k}\}$ and $\{a_{2,\ell k}\}$ that form left-stochastic matrices $A_1$ and $A_2$ and satisfy:
\begin{align}
		a_{1,\ell k} &= 0,\quad\quad \textrm{when\ } \ell \notin \mathcal{N}_k\\
		a_{2,\ell k} &= 0,\quad\quad \textrm{when\ } \ell \notin \mathcal{N}_k
\end{align}
In \eqref{eq:diff_A}, we already replaced the true gradient vector, $\nabla_w J_k(\cdot)$, with an approximation $\widehat{\nabla_w} J_k(\cdot)$, usually evaluated from instantaneous data realizations. For this reason, $\bphi$, $\bpsi$, $\bzeta$, and $\w$ in \eqref{eq:diff_C1}--\eqref{eq:diff_C2} are denoted in boldface to highlight that they are now random variables. In order to recover the ATC algorithm, we set $A_1 = I_N$ and $A_2=A$ and to recover the CTA algorithm we set $A_1=A$ and $A_2=I_N$.

Since the iterate $\w_{k,i}$ generated by \eqref{eq:diff_C2} is random, we shall measure performance by examining the average squared distance between $\w_{k,i}$ and $w^\star$:
\begin{align}
	\limsup_{i\rightarrow\infty}\ \E \|w^\star - \w_{k,i}\|^2
	\label{eq:desired_metric}
\end{align}
Now, using the optimal solution $w^o(\eta)$ of \eqref{eq:global_cost_approx} we can write:
\begin{align}
	\!\!\limsup_{i\rightarrow\infty}\ \!\!\E \|w^\star \!\!\!-\! \w_{k,i}\|^2 \!&=\! \limsup_{i\rightarrow\infty}\ \!\!\E \|\!w^\star \!\!\!-\! w^o\!(\eta) \!\!+\!\! w^o\!(\eta) \!\!-\! \w_{k,i}\|^2 \nonumber\\
	&\!\!\!\!\!\!\!\!\!\!\!\!\!\!\!\!\!\!\!\!\!\!\!\!\!\!\!\!\!\!\!\!\!\!\!\!\!\!\!\!\!\!\!\!\!\!\!\leq 2 \underset{\textrm{Approximation Error}}{\underbrace{\|w^\star - w^o(\eta)\|^2}} + 2 \limsup_{i\rightarrow\infty}\ \E \|w^o\!(\eta) - \w_{k,i}\|^2 \label{eq:error_expansion_1}
\end{align}
We will see later that the approximation error $\|w^\star - w^o(\eta)\|^2$ can be driven to arbitrarily small values as $\eta \rightarrow \infty$. This agrees with the intuition from Sec.~\ref{sec:penalty_method}. After we establish this fact, we shift our attention towards characterizing the second term of the upper bound in \eqref{eq:error_expansion_1} in order to assess how small \eqref{eq:desired_metric} is. 

We now introduce the necessary assumptions for studying the performance of the diffusion strategies and explain how they arise and where they are used in the analysis. These conditions are of the same nature as assumptions regularly used in the broad distributed optimization literature, as indicated by the references given below in the explanations.
\subsection{Main Assumptions}
\begin{assumption}[Feasible problem]
	\label{ass:uniqueness_of_w^star}
	Problem \eqref{eq:orig_dist_problem} is feasible and, therefore, a minimizer $w^{\star}\in \mathbb{W}$ exists.\qed
\end{assumption}
\noindent This is a logical assumption and it simply states that the set $\mathbb{W}_1\cap\ldots\cap\mathbb{W}_N$ is non-empty. This situation is common when analyzing barrier and penalty methods \cite[p.~561]{cvx_book} for solving convex optimization problems.
\begin{assumption}[Individual costs]
\label{ass:Hessian_cost}
	Each cost function $J_k(w)$ has a Hessian matrix that is bounded from below, i.e., there exist $\{\lambda_{k,\min} > 0\}$ such that, for each $k=1,\ldots,N$:
	\begin{align}
		\nabla_w^2 J_k(w) \geq \lambda_{k,\min} I_M \label{eq:strongly_convex_condition1}
	\end{align}
	Furthermore, since the individual costs $J_k(w)$ are strongly convex, there exist $\lambda_{k,\max}>0$ such that 
	\begin{align}
		\nabla_w^2 J_k(w) \leq \lambda_{k,\max} I_M \label{eq:bounded_hessian}
	\end{align}	
	\qed
\end{assumption}
Notice that the bounded Hessian assumption \eqref{eq:bounded_hessian} is a relaxation of the bounded gradient assumption used in earlier studies, e.g., in \cite{yan,Nedic}. Assumption \ref{ass:Hessian_cost} allows the case in which some of the nodes may not possess a cost function that depends on $w$ at all, but that they only enforce constraints with regularization. Observe that when \eqref{eq:strongly_convex_condition1} holds, we have the following facts.
\begin{fact}[Uniqueness of $w^\star$]
	\label{fact:uniqueSolution}
	When Assumption \eqref{ass:uniqueness_of_w^star} and \eqref{eq:strongly_convex_condition1} hold, the optimizer $w^\star$ of \eqref{eq:orig_dist_problem} is unique \cite[p.~217]{fletcher1987practical}. \qed
\end{fact}
\begin{fact}[Uniqueness of $w^o(\eta)$]
	\label{fact:approx_unique}
	When \eqref{eq:strongly_convex_condition1} holds, the optimizer $w^o(\eta)$ of \eqref{eq:global_cost_approx} is unique for any $\eta\geq 0$.\qed
\end{fact}
	Observe that Fact \ref{fact:approx_unique} does not require the existence of $w^\star$ (Assumption \ref{ass:uniqueness_of_w^star}) in order for $w^o(\eta)$ to be unique---since in this case, $w^o(\eta)$ will be infeasible in terms of $\mathbb{W}$ even as $\eta\rightarrow\infty$, and thus not meaningful. Fact \ref{fact:uniqueSolution} follows from Assumptions \ref{ass:uniqueness_of_w^star}--\ref{ass:Hessian_cost} since strict convexity (which is guaranteed by strong convexity) of the objective function, and the existence of an optimizer, guarantee uniqueness of the optimizer \cite[p.~217]{fletcher1987practical}. The reason Fact \ref{fact:approx_unique} follows from Assumption \ref{ass:Hessian_cost} is that the aggregate cost in \eqref{eq:global_cost_approx} will be strongly-convex. 

We also require the Hessian matrices of the penalty functions with respect to $w$ to be bounded from above, but not necessarily from below (they are obviously nonnegative definite since the penalty functions are convex).
\begin{assumption}[Penalty functions]
\label{ass:Hessian_penalty}
The Hessian matrix of each penalty function $p_k(w)$, with respect to $w$, is upper bounded, i.e., 
\begin{align}
	\nabla_w^2 p_k(w) \leq \lambda_{k,\max}^p I_M
\end{align}
where $\lambda^p_{k,\max} > 0$ for all $u = 1,\ldots,U_k$, $l=1,\ldots,L_k$, and $k=1,\ldots,N$. Furthermore, since the penalty functions are convex, their Hessian matrices are nonnegative definite.\qed
\end{assumption}
\begin{assumption}[Combination matrices]
\label{ass:primitive}
	The combination matrix $A$ in the penalized ATC or CTA implementation is primitive and doubly-stochastic.\qed
\end{assumption}
Since in our unified framework \eqref{eq:diff_C1}-\eqref{eq:diff_C2}, either $A_1$ or $A_2$ is the identity matrix, then Assumption \ref{ass:primitive} is equivalent to requiring that the product matrix $A=A_1A_2$ is primitive and doubly-stochastic. A doubly-stochastic matrix $A$ is one that satisfies $A^{\T} \mathds{1}=\mathds{1}$ and $A\mathds{1}=\mathds{1}$ so that the entries on each of its columns and on each of its rows add up to one. The widely used Metropolis weights \cite{book_chapter,boyd2004fastest,Xiaochuan} satisfy Assumption \ref{ass:primitive} and can be computed in a distributed manner:
\begin{align}
	a_{\ell k} = \begin{cases}
			\min\left(\frac{1}{|\mathcal{N}_\ell|}, \frac{1}{|\mathcal{N}_k|}\right), & \ell \in \mathcal{N}_k, \ell \neq k\\
			1 - \displaystyle{\sum_{j\in {\cal N}_k\backslash\{k\}}} a_{j k},& \ell = k\\
			0, & \textrm{otherwise}
	\end{cases}
	\label{eq:Metropolis_DS}
\end{align}
where the notation $|{\cal N}_{k}|$ denotes the degree of node $k$ or the number of its neighbors. The primitive condition on $A$ is satisfied by any connected network with at least one self-loop (i.e., at least one $a_{k,k}>0$) \cite{book_chapter}. This situation is common in practice where networks tend to be connected and at least one node has some level of trust in its own data. 
\begin{assumption}[Gradient noise model]
	\label{ass:noiseModeling}
	We model the perturbed gradient vector as:
	\begin{align}
		\widehat{\nabla_w} J_k(\w) &= \nabla_w J_k(\w) + \v_{k,i}(\w)
		\label{eq:grad_model}
	\end{align}
	where, conditioned on the past history of the iterates $\mathcal{H}_{i-1} \triangleq \{\w_{k,j}:k = 1 , \ldots, N\ \mathrm{and}\ j \leq i-1\}$, the gradient noise $\v_{k,i}(\w)$ is assumed to satisfy:
\begin{align}
		&\mathbb{E}\{\v_{k,i}(\w) | \mathcal{H}_{i-1}\} = 0 \label{eq:zero_mean_noise}\\
		&\mathbb{E}\|\v_{k,i}(\w)\|^2 \leq \alpha \mathbb{E}\|\w\|^2 + \sigma_v^2 \label{eq:variance_noise}
	\end{align}
	for some $\alpha\geq 0$, $\sigma_v^2 \geq 0$, and where $\w \in \mathcal{H}_{i-1}$.\hfill$\blacksquare$
\end{assumption}
\noindent Models similar to \eqref{eq:zero_mean_noise}--\eqref{eq:variance_noise} are also used in the works by \cite{tsitsiklis1986distributed,Polyak} on distributed algorithms --- see the explanation in \cite{Jianshu_common_wo}. We are now ready to state our main results. We delay most of the proofs to the appendices to simplify the exposition.
\section{Main Convergence Result}
First, we characterize the distance between the optimizer of the augmented cost function \eqref{eq:J^approx}, $w^o(\eta)$, and the optimizer of the original optimization problem \eqref{eq:orig_dist_problem}, $w^\star$. This distance appears in the first term of \eqref{eq:error_expansion_1}. Therefore, in order to show that the right-hand-side of \eqref{eq:error_expansion_1} can be made arbitrarily small, we must first show that this distance can be made arbitrarily small by choosing $\eta$ appropriately. For convenience, we introduce the compact notation:
\begin{align}
w^o(\infty)\triangleq \lim_{\eta\rightarrow\infty} w^o(\eta) \label{eq:w(infty,tau)}
\end{align}
\begin{theorem}[Approaching optimal solution]
\label{thm:approx_error}
	Under Assumptions \ref{ass:uniqueness_of_w^star}, \ref{ass:Hessian_cost}, it holds that:
	\begin{align}
		 \|w^\star - w^o(\infty)\| = 0
		 \label{eq:approx_error_norm}
	\end{align}
	so that $w^o(\infty)$ is feasible and optimal. 
\end{theorem}
\begin{proof}
	Since $J^\textrm{glob}_{\eta}(w)$ is strongly convex, we have that for any point $w \in \mathbb{R}^M$, the distance from the optimizer $w^o(\eta)$ is bounded by \cite[p.~460]{cvx_book}:
	\begin{align}
		\|w^o(\eta)-w\| \leq \frac{2}{\lambda_{\min}} \|\nabla_w J^\textrm{glob}_{\eta}(w)\| \label{eq:dist_wo-w1}
	\end{align}
	where $\lambda_{\min} = \min_k \{\lambda_{k,\min}\}$ as defined in Assumption \ref{ass:Hessian_cost}. It is possible to obtain an upper bound in \eqref{eq:dist_wo-w1} that is independent of $\eta$ as follows. Since we are free to pick $w$, we let $w=w^\star$, where $w^\star \in \mathbb{W}$ by Assumption \ref{ass:uniqueness_of_w^star} to obtain
	\begin{align}
		\|w^o(\eta)-w^\star\| \leq \frac{2}{\lambda_{\min}} \|\nabla_w J^\textrm{glob}_\eta(w^\star)\|
	\end{align}
	Recalling \eqref{eq:J^approx}, we have that
	\begin{align}
		\nabla_w J^\textrm{glob}_\eta(w^\star) = \nabla_w J^\textrm{glob}(w^\star)+ \eta \sum_{k=1}^N p_k(w^\star)
	\end{align}
	but since by construction, $p_k(w') = 0$ when $w' \in \mathbb{W}$, we have that 
	\begin{align}
		\nabla_w J^\textrm{glob}_\eta(w^\star) = \nabla_w J^\textrm{glob}(w^\star)
	\end{align}
	and since $\|w^o(\eta)\| \leq \|w^o(\eta) - w^\star\| + \|w^\star\|$, we obtain
	\begin{align}
		\|w^o(\eta)\| \leq \frac{2}{\lambda_{\min}} \|\nabla_w J^\textrm{glob}(w^\star)\| + \|w^\star\| < \infty \label{eq:final_wo_bound_indep_eta}
	\end{align}
	The upper bound in \eqref{eq:final_wo_bound_indep_eta} is independent of $\eta$ and is also finite since $J^\textrm{glob}(w)$ is a continuous function in $w$. To obtain \eqref{eq:approx_error_norm}, we appeal to Theorem 9.2.2 of \cite{Bazaraa} by noting that $w^o(\eta) \in \mathbb{B}$, where $\mathbb{B} \subset \mathbb{R}^M$ is the compact set \cite[p.~2--3,188]{SteinReal2005}
	\begin{align}
	\mathbb{B} = \left\{w : \|w\| \leq \frac{2}{\lambda_{\min}} \|\nabla_w J^\textrm{glob}(w^\star)\|+\|w^\star\|\right\}
	\end{align}
	from which we can conclude \eqref{eq:approx_error_norm}.
\end{proof}
We now turn our attention to the convergence of the distributed algorithm.
\begin{theorem}[Convergence condition]
	\label{thm:alg_error}
	Let Assumptions \ref{ass:Hessian_cost}, \ref{ass:Hessian_penalty}, \ref{ass:primitive}, and \ref{ass:noiseModeling} hold. Then, the diffusion strategy \eqref{eq:diff_C1}--\eqref{eq:diff_C2} converges for sufficiently small positive step-sizes, namely, for step-sizes that satisfy 
	{
	\begin{align}
	\mu &< \min_{1\leq k \leq N}\left\{\frac{2\lambda_{k,\max}}{\lambda_{k,\max}^2+2\alpha},\frac{2\lambda_{k,\min}}{\lambda_{k,\min}^2+2\alpha}, \frac{2}{\eta \cdot\lambda_{k,\max}^p}\right\}
	\end{align}	}
	Specifically, it holds that for small $\mu$ 
	\begin{align}	
		\limsup_{i\rightarrow\infty}\  \E\|w^o(\eta)-\w_{k,i}\|^2	&\leq O(\mu) + O((\eta\cdot\mu)^2) \label{eq:result_SS_small_mu}
	\end{align}
	so that
	\begin{align}
		\lim_{\mu\rightarrow 0} \limsup_{i\rightarrow\infty}\  \E\|w^o(\eta)-\w_{k,i}\|^2	&= 0. \label{eq:limsup}
	\end{align}
	
\end{theorem}
\begin{proof}
	See Appendix \ref{app:alg_error}.
\end{proof}
Theorem \ref{thm:alg_error} states that the expected squared distance between $\w_{k,i}$ at each node and $w^o(\eta)$ is on the order of $\mu$ or $(\eta\cdot\mu)^2$, whichever is larger. This implies that when the step-size is chosen to be sufficiently small, the expected error can be made arbitrarily small as long as $\eta\in O(1/\mu)$. 

We conclude from \eqref{eq:approx_error_norm} and \eqref{eq:limsup} that
\begin{align}
	\lim_{{\scriptsize{\begin{array}{c}
	\mu\!\rightarrow\! 0 \\ 
	\eta\!\rightarrow\!\infty
	\end{array}}}} \limsup_{i\rightarrow\infty}\ \E\|w^\star \!-\! \w_{k,i}\|^2 = 0 \label{eq:limit_mu_0}
\end{align}
which may be simplified if we choose the parameter $\eta$ in terms of $\mu$ as follows:
\begin{align}
	\eta \triangleq \mu^{-\theta},\quad 0< \theta < 1
	\label{eq:condition_omega}
\end{align}
Then, we have that:
\begin{align}
\boxed{
	\lim_{\mu\rightarrow 0} \limsup_{i\rightarrow\infty}\ \E\|w^\star - \w_{k,i}\|^2 = 0
	}
\end{align}
We conclude that the diffusion strategy \eqref{eq:diff_C1}--\eqref{eq:diff_C2} effectively solves \eqref{eq:orig_dist_problem} in a fully distributed manner with progressively improving estimates of the optimizer as $\mu\rightarrow 0$. 
In addition, the diffusion algorithm, which utilizes a constant step-size, is capable of tracking varying constraint sets and will continue to track the true optimizer $w^\star$ as the convex constraint sets, $\mathbb{W}_k$, and cost functions, $J_k(w)$, drift, as illustrated next.
\begin{figure*}[tH!]
	\centering
	\includegraphics[width=14.5cm,height=5.7cm]{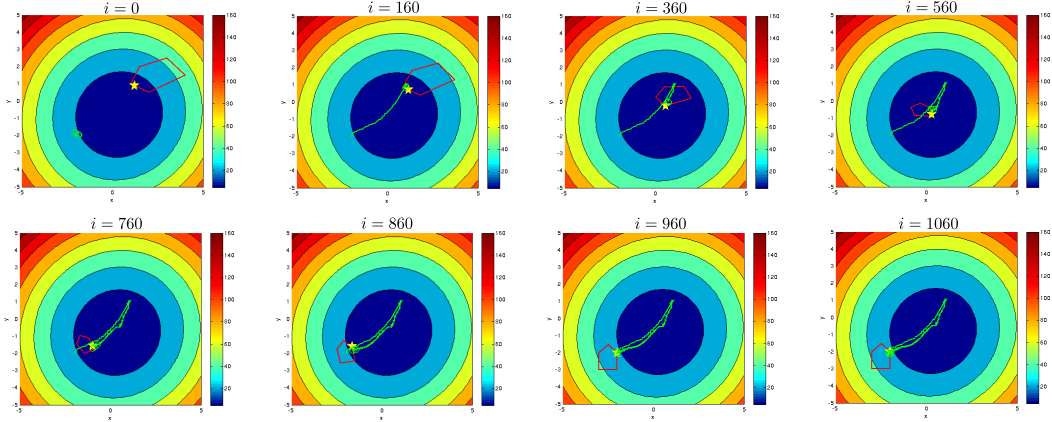}\vspace*{-1\baselineskip}
	\caption{The star indicates the location of the optimal minimizer, $w^\star_i$, which is allowed to drift in this simulation to illustrate the tracking ability of the algorithm. The green curve illustrates the location of the estimates by the nodes; it is seen from the second plot from the left in the first row corresponding to $i=160$ that this curve converges to the minimizer location. As the constraint set begins to change starting at $i = 160$, we notice that the estimates are able to track the minimizer even as the feasible region shrinks and changes with time.}\vspace{-1\baselineskip}
	\label{fig:sim}
\end{figure*}

\section{Simulation}
We consider a distributed optimization problem with $N=5$ nodes in order to illustrate the trajectory of the solutions clearly. Each node is associated with the mean-square-error cost $J_k(w) = \E (\d_k(i)-\h_{k,i}^\T w)^2$, where the desired signal $\d_k(i)$ is related to some unknown model $\overline{w}$ via the linear regression model:
\begin{align}
	\d_k(i) \triangleq \h_{k,i}^\T \overline{w} + \v_k(i)
\end{align}

\noindent To illustrate adaptation and tracking ability, we introduce a single moving hyper-plane per node of the form
\begin{align}
	g_{k,i}(w) \triangleq b_{k,i}^\T w - z_k(i)
\end{align}
where $\{b_{k,i},z_k(i)\}$ are allowed to change with $i$. If we define the matrix $B_i = \textrm{col}\{b_{1,i}^\T,b_{2,i}^\T,\ldots,b_{N,i}^\T\}$ and the vector $z_i=\textrm{col}\{z_1(i),\ldots,z_N(i)\}$, then we have that the global optimization problem is of the form:
\begin{alignat}{3}
	&\min_w\ && \sum_{k=1}^N \E (\d_k(i)-\h_{k,i}^\T w)^2 \label{eq:simulation_orig_problem}\\
	&\textrm{subject\ to\ }\ \ && B_i w - z_i \preceq 0 \nonumber
\end{alignat}
While the projections associated with the distributed solution of this problem may be solved analytically, this setup allows us to demonstrate the tracking ability of the proposed algorithm, which does not rely on the use of projections. We let the inequality constraints drift with time and we track the progress of the algorithm as the estimates at each of the nodes move towards to true optimizer $w^\star_i$ of \eqref{eq:simulation_orig_problem} at time $i$. The statistical distributions associated with $\h_{k,i}$ and $\v_{k}(i)$ remain fixed for the duration of the simulation---and, therefore, $J_k(w)$ is fixed in this simulation while the constraints are drifting. While this need not be the case in general, and the diffusion algorithm will handle the non-stationary cost function scenario as well, keeping the cost function fixed facilitates the illustration of the results.

The variance of the noise $\v_k(i)$ is chosen randomly for each node so that $\sigma_{v,k}^2 \sim U(0,1)$, where $U(0,1)$ denotes a uniform distribution on the range $[0,1)$. The covariance matrices $\E \h_{k,i} \h_{k,i}^\T = R_{h,k}$ are generated as $R_{h,k} = Q_k \Lambda_k Q_k^\T$ where $Q_k$ is a randomly generated orthogonal matrix and $\Lambda_k$ is a diagonal matrix with random elements so that $(\Lambda_{k})_{l,l} \sim U(0,1)$. The model vector $\overline{w} \in \mathbb{R}^2$ is chosen randomly for the simulation. The constraint set is also initialized randomly, morphs and moves as time progresses throughout the simulation. A stepsize of $\mu=0.01$ is chosen with $\eta = 30$. The combination weights used throughout the simulation are based on the Metropolis rule \eqref{eq:Metropolis_DS}. The penalty function $\delta(x) = \sqrt{x^2+\rho^2}$ was used in the simulation with $\rho = 0.01$. Figure \ref{fig:sim} illustrates the evolution of the estimates across the nodes as time progresses. We observe that the nodes are attracted towards the feasible region from their initial position and quickly converge towards the true optimizer $w^\star_i$, which is initially stationary. As the constraint set begins to change after $i = 160$, we notice that each node's estimate of the optimizer changes and tracks $w^\star_i$ even as the feasible region shrinks and continues to move throughout the simulation. The green line corresponds to the average trajectory of the nodes' estimates throughout the simulation.
\section{Conclusion}
In this work, we developed a distributed optimization strategy based on diffusion adaptation that allows a network of agents to solve a constrained convex problem in which the objective function is the aggregate sum of individual convex objective functions distributed across the nodes. The constraint set is the intersection of convex constraints at each node. The algorithm does not require the agents to know about other constraints besides their own. We showed that through local interactions, the network is able to approach the desired global minimizer to arbitrarily good accuracy levels. The convergence analysis was performed in the stochastic setting in which the gradient vectors of the individual cost functions may not be available at each node and are approximated in the presence of gradient noise. 

\appendices
{{

\section{Proof of Theorem \ref{thm:alg_error}}
\label{app:alg_error}
In this section, we analyze how well the diffusion strategy \eqref{eq:diff_C1}-\eqref{eq:diff_C2} approaches the optimal solution $w^o(\eta)$ of the augmented cost \eqref{eq:global_cost_approx}. We examine this performance in terms of the mean squared error measure, $\E\|w^o(\eta)-\w_{k,i}\|^2$, in the presence of gradient noise, as modeled by Assumption \ref{ass:noiseModeling}. We extend the energy analysis framework developed in \cite{Jianshu_diff_wo} to handle constrained optimization. Compared with the diffusion strategy studied in \cite{Jianshu_diff_wo}, however, the models there did not incorporate projection steps similar to \eqref{eq:ATC_P} and \eqref{eq:CTA_P}.  When these steps are incorporated, certain differences arise in the analysis that require attention (e.g., some symmetry properties present in the analysis of \cite{Jianshu_diff_wo} are lost in the current context and need to be addressed). We first show that the diffusion strategy, in the absence of gradient noise, converges and has a fixed-point. Subsequently, we analyze the distance between this point and the vectors $w^o(\eta)$ and $\w_{k,i}$ in the mean-square-sense.

\subsection{Existence of Fixed Point}
At each iteration, we can view the diffusion strategy \eqref{eq:diff_C1}--\eqref{eq:diff_C2} as a mapping from the vectors $\{\w_{k,i-1}\}$ to the vectors $\{\w_{k,i}\}$ or, more generically, as a mapping from some block vector $x$ to another block vector $w$. Thus, let $x=\mbox{\rm col}\{x_1,x_2,\ldots,x_N\}$ denote a block vector with sub-vectors $x_k$ of size $M\times 1$. Let also $P[x] \triangleq \mbox{\rm col}\{\|x_1\|^2,\|x_2\|^2,\ldots,\|x_N\|^2\}$.  Then, we observe that given any two input vectors $x^1, x^2 \in \mathbb{R}^{MN}$, the resulting updated vectors $w^1$ and $w^2$ are given by 
\begin{subequations}
\begin{align}
	w^1 &= (A_2^\T \otimes I_M) \psi^1,\quad w^2 = (A_2^\T \otimes I_M) \psi^2 \label{eq:psi_to_w}
\end{align}
where the intermediate vectors $\psi^1$ and $\psi^2$ are constructed as follows in terms of other intermediate block vectors $\{\zeta^1,\zeta^2,\phi^1,\phi^2\}$:
\begin{align}
	\!\!\psi^1 \!&=\!\! \left[\!\!\!\!\begin{array}{c}
	\zeta_1^1 \!-\! \mu\eta \nabla_w p_1(\zeta_1^1)\\ 
	\vdots\\
	\zeta_N^1 \!-\! \mu\eta \nabla_w p_N(\zeta_N^1)
	\end{array} \!\!\!\!\right]\!,\!
	\psi^2 \!\!=\!\! \left[\!\!\!\!\begin{array}{c}
	\zeta_1^2 \!-\! \mu\eta \nabla_w p_1(\zeta_1^2)\\ 
	\vdots\\
	\zeta_N^2 \!-\! \mu\eta \nabla_w p_N(\zeta_N^2)
	\end{array} \!\!\!\!\right] \label{eq:zeta_to_psi}\\
	\!\!\zeta^1 \!&=\!\! \left[\!\!\!\!\begin{array}{c}
	\phi_1^1 \!-\! \mu \nabla_w J_1(\phi_1^1)\\ 
	\vdots\\
	\phi_N^1 \!-\! \mu \nabla_w J_N(\phi_N^1)
	\end{array} \!\!\!\!\right]\!\!,\ \!\!
	\zeta^2 \!\!=\!\!\! \left[\!\!\!\!\begin{array}{c}
	\phi_1^2 \!-\! \mu \nabla_w J_1(\phi_1^2)\\ 
	\vdots\\
	\phi_N^2 \!-\! \mu \nabla_w J_N(\phi_N^2)
	\end{array} \!\!\!\!\right] \label{eq:phi_to_zeta}\\
	\phi^1 &= (A_1^\T \otimes I_M) x^1,\quad \phi^2 = (A_1^\T \otimes I_M) x^2 \label{eq:x_to_phi}
\end{align}
\end{subequations}
We now verify that the mapping $x \mapsto w$ is a contraction for sufficiently small step-sizes. Indeed, using the sub-multiplicative property of the block-maximum norm and the fact that $A_1$ and $A_2$ are left-stochastic \cite{book_chapter}, we conclude from \eqref{eq:psi_to_w} and \eqref{eq:x_to_phi}:
\begin{align}
	\|w^1-w^2\|_{b,\infty} &\leq \|\psi^1-\psi^2\|_{b,\infty} \label{eq:psi_diff_expression}\\
	\|\phi^1-\phi^2\|_{b,\infty} &\leq \|x^1 - x^2\|_{b,\infty}
\end{align}
Now, we can bound the quantity $\|\psi^1-\psi^2\|_{b,\infty}$ by appealing to the mean-value theorem \cite[p.~24]{Polyak} to write:
	\begin{align}
		\nabla_w& p_k(\zeta_k^1) -  \nabla_w p_k(\zeta_k^2) \!= \nonumber\\
		& \left(\int_0^1 \nabla_w^2 p_k(\zeta_k^2 + t (\zeta_k^1-\zeta_k^2)) dt \right)(\zeta_k^1-\zeta_k^2) \label{eq:mean-value-theorem-Hessian}
	\end{align}
	from which we conclude that 
\begin{align}
	&\|\psi^1-\psi^2\|_{b,\infty} \nonumber\\
	&\leq \max_{1\leq k\leq N} \left\|I_M\!-\!\mu\eta\!\! \int_0^1 \nabla_w^2 p_k(\zeta_k^2 + t (\zeta_k^1-\zeta_k^2)) dt \right\|\!\cdot\!\! \|\zeta_k^1-\zeta_k^2\| \label{eq:first_eq_psi_diff}
\end{align}
Now, due to Assumption \ref{ass:Hessian_penalty}, we have that 
\begin{equation}
	\left\|\!I_M\!\!-\!\!\mu\eta \!\! \int_0^1\!\!\!\!\! \nabla_w^2 p_k(\zeta_k^2 + t (\zeta_k^1-\zeta_k^2)) dt \right\| \!\leq\! \max\{|1\!-\!\mu \eta \lambda_{k,\max}^p|,1\!\} \label{eq:bound_hessian_p}
\end{equation}}
The bound on the right-hand side of \eqref{eq:bound_hessian_p} can be guaranteed to be at most one when
\begin{equation}
\begin{aligned}
	    0 &\leq \mu\eta \leq \min_{1\leq k \leq N} \left\{ \frac{2}{\lambda_{k,\max}^p} \right\}
\end{aligned}
\label{eq:condition_nu_1}
\end{equation}
so that
\begin{align}
	\|\psi^1-\psi^2\|_{b,\infty} \leq \|\zeta^1-\zeta^2\|_{b,\infty} \label{eq:last_eq:psi_diff}
\end{align}
In a similar manner to \eqref{eq:mean-value-theorem-Hessian}--\eqref{eq:first_eq_psi_diff}, we can verify that 
\begin{align}
		&\|\zeta^1-\zeta^2\|_{b,\infty} \nonumber\\
	&\leq \max_{1\leq k\leq N} \left\|I_M\!-\!\mu\!\! \int_0^1 \!\!\!\!\nabla_w^2 J_k(\phi_k^2 + t (\phi_k^1-\phi_k^2)) dt \right\|\!\cdot\! \|\phi_k^1-\phi_k^2\| \label{eq:first_eq_psi_diff2}
\end{align}
and due to Assumption \ref{ass:Hessian_cost}, 
\begin{align}
	\lambda_{k,\min} I_M \!\leq\! \int_0^1\!\!\!\! \nabla_w^2 J_k(\phi_k^2 + t (\phi_k^1-\phi_k^2)) dt \!\leq\! \lambda_{k,\max} I_M
\end{align}
It follows that 
\begin{align}
	\|\zeta^1-\zeta^2\|_{b,\infty} &\leq \gamma \cdot \|\phi^1-\phi^2\|_{b,\infty}\\
	\gamma &\triangleq \max_{1\leq k\leq N} \left\{\gamma_k\right\} \label{eq:gamma}\\
	\gamma_k &\triangleq \max\{|1-\mu \lambda_{k,\min}|,|1-\mu \lambda_{k,\max}|\} \label{eq:gamma_k}
\end{align}
and $\gamma_k$ satisfies $0 \leq \gamma_k < 1$ when
\begin{align}
	0 < \mu < \min_{1\leq k \leq N}\left\{\frac{2}{\lambda_{k,\max}}\right\} \label{eq:bound_stability2}
\end{align}
Combining the previous results together we arrive at 
\begin{align}
	\|w^1-w^2\|_{b,\infty} \leq \gamma \|x^1 - x^2\|_{b,\infty}
\end{align}
for $\gamma < 1$ when \eqref{eq:condition_nu_1} and \eqref{eq:bound_stability2} are satisfied. 
\noindent \begin{remark}\label{rem:remark2} \emph{It is the above argument that relies on the requirement that all individual costs are strongly-convex so that all the $\lambda_{k,\min}$ are strictly positive and each $\gamma_k$ can be made strictly less than one. If we relax the strong convexity assumption and require only at least one of the individual costs to be strongly convex, then the above argument needs to be adjusted as done in \cite{chen2012limiting}; nevertheless, the conclusion of Theorem \ref{thm:alg_error} will continue to hold, namely, results \eqref{eq:result_SS_small_mu} and \eqref{eq:limit_mu_0} for sufficiently small step-sizes.} \qed
\end{remark}
We conclude that the diffusion mapping $x \mapsto w$ is a contraction mapping for sufficiently small step-sizes. By the Banach fixed point theorem \cite[pp.~299--303]{kreyszig1989introductory}, this mapping will have a unique fixed point, $w_\infty$. Observe that this fixed point is \emph{not} $\mathds{1}_N \otimes w^o(\eta)$. However, since we wish to study the rightmost term in \eqref{eq:error_expansion_1}, or, equivalently:
\begin{align}
	\limsup_{i\rightarrow\infty}\ \E\|\mathds{1}_N \otimes w^o(\eta)- \w_{i}\|^2
\end{align}
where $\w_i \triangleq \textrm{col}\{\w_{1,i},\ldots,\w_{N,i}\}$, we will decompose the above squared distance into two parts: (1) the expected squared distance from $w_\infty$ to $\w_{i}$, and (2) the squared distance from $w^o(\eta)$ to $w_\infty$ (the bias of the algorithm):
\begin{align}
	\E\|\mathds{1} \!\otimes\! w^o(\eta)\!-\!\w_{i}\|^2 &= \E\|\mathds{1}\! \otimes\! w^o(\eta)-w_\infty+w_\infty-\w_{i}\|^2\nonumber\\
	&\!\!\!\!\!\!\!\!\!\!\!\!\!\!\!\!\!\!\!\leq 2\mathds{1}_N^\T \E P[\!\w_i\!-\!w_\infty\!]\!+\!\!2\|\mathds{1} \!\otimes\! w^o\!(\eta)\!\!-\!\!w_\infty\!\|^2 \label{eq:triangle_inequality}
\end{align}}
In order to complete the study, we first examine the quantity $\E P[\w_i\!-\!w_\infty]$ and assess the size of the right-most term. 
\subsection{Mean-Square-Distance to Fixed Point}
We introduce the vectors $\phi_{\infty}$, $\psi_{\infty}$, $\zeta_{\infty}$, and the fixed-point $w_\infty$ and their respective blocks $\phi_{k,\infty}$, $\psi_{k,\infty}$, $\zeta_{k,\infty}$, and $w_{k,\infty}$ obtained by letting $x^1 = w_\infty$ in \eqref{eq:psi_to_w}--\eqref{eq:x_to_phi}. The mean-square-error between the iterates $\boldsymbol{\phi}_{k,i-1}$ and $\w_{k,i}$ and their respective limit points in the noiseless recursion are bounded using Jensen's inequality \cite[p.~77]{cvx_book}
\begin{align}
	\E \|w_{k,\infty} - \w_{k,i}\|^2 &\leq \sum_{\ell = 1}^N a_{2,\ell k} \E\|\psi_{k,\infty} - \boldsymbol{\psi}_{\ell,i}\|^2 \label{eq:final_relation_w}\\
	\E \|\phi_{k,\infty} - \boldsymbol{\phi}_{k,i-1}\|^2 &\leq \sum_{\ell = 1}^N a_{1,\ell k} \E\|w_{k,\infty} - \w_{\ell,i-1}\|^2 \label{eq:final_relation_phi}
\end{align}
We also have that 
\begin{align}
	&\E \|\psi_{k,\infty} \!-\! \boldsymbol{\psi}_{k,i}\|^2 \!=\! \E\|\zeta_{k,\infty} \!-\! \boldsymbol{\zeta}_{k,i}\|^2_{\boldsymbol{\Omega}_{k,i}}
\end{align}
where
\begin{align}
	\boldsymbol{\Omega}_{k,i} \triangleq \left(I_M - \int_0^1 \nabla_w^2 p_k(\zeta_{k,\infty} - t (\zeta_{k,\infty}-\boldsymbol{\zeta}_{k,i})) dt\right)^2
\end{align}
But due to \eqref{eq:bound_hessian_p}--\eqref{eq:condition_nu_1}, we have that 
\begin{align}
	\E \|\psi_{k,\infty} \!-\! \boldsymbol{\psi}_{k,i}\|^2  \leq 
\E\|\zeta_{k,\infty} \!-\! \boldsymbol{\zeta}_{k,i}\|^2 \label{eq:final_relation_psi}
\end{align}
when \eqref{eq:condition_nu_1} is satisfied. Moreover, the mean-square-error between $\zeta_{k,\infty}$ and $\boldsymbol{\zeta}_{k,i}$ can be bounded by
\begin{align}
	&\E \|\zeta_{k,\infty} \!-\! \boldsymbol{\zeta}_{k,i}\|^2 \nonumber\\
	&\stackrel{(a)}{=}\! \E\|\phi_{k,\infty} \!-\! \boldsymbol{\phi}_{k,i-1}\|^2_{\boldsymbol{\Sigma_{k,i-1}}} \!\!\!+\!\! \mu^2 \E \|\v_{k,i}(\boldsymbol{\phi}_{k,i-1})\|^2 \nonumber\\
	&\stackrel{(b)}{\leq} \E \|\phi_{k,\infty}-\boldsymbol{\phi}_{k,i-1}\|^2_{\boldsymbol{\Sigma_{k,i-1}}} \!\!\!+\! \mu^2 \left(\alpha \E\|\boldsymbol{\phi}_{k,i-1}\|^2 \!+\! \sigma_v^2\right) \nonumber\\
	&=  \E \|\phi_{k,\infty}-\boldsymbol{\phi}_{k,i-1}\|^2_{\boldsymbol{\Sigma_{k,i-1}}} +\nonumber\\
	&\quad\ \mu^2 \left(\alpha \E\|w^o(\eta)\!-\!\phi_{k,\infty}\!+\!\phi_{k,\infty}\!-\!\boldsymbol{\phi}_{k,i-1}-w^o(\eta)\|^2 \!+\! \sigma_v^2\right) \nonumber\\
	&\leq \E \|\phi_{k,\infty}-\boldsymbol{\phi}_{k,i-1}\|^2_{\boldsymbol{\Sigma_{k,i-1}}} \!\!\!+\! 2\mu^2 \alpha \E\|w^o(\eta)\!-\!\phi_{k,\infty}\|^2+ \nonumber\\
	&\quad\ 2\mu^2 \alpha\|\phi_{k,\infty}\!-\!\boldsymbol{\phi}_{k,i-1}\|^2 \!+\! \mu^2 (2\alpha \|w^o(\eta)\|^2 + \sigma_v^2) \label{eq:zeta_diff}
\end{align}
where step $(a)$ can be obtained via an argument similar to \eqref{eq:mean-value-theorem-Hessian}, step $(b)$ is due to Assumption \ref{ass:noiseModeling} and $\boldsymbol{\Sigma}_{k,i-1} \triangleq (I_M - \mu \H_{k,i-1})^2$, where $\H_{k,i-1}$ is defined as:
\begin{align}
	\H_{k,i-1} \triangleq \int_0^1 \nabla^2_w J_k(\phi_{k,\infty} - t(\phi_{k,\infty} - \boldsymbol{\phi}_{k,i-1})) dt
\end{align}
Now, due to Assumption \ref{ass:Hessian_cost}, we have that $0 \leq \boldsymbol{\Sigma}_{k,i-1} \leq \gamma_k^2 I_M$, where $\gamma_k$ is defined in \eqref{eq:gamma_k}. Furthermore, from \eqref{eq:diff_C1} it is possible to bound $\|w^o(\eta)\!-\!\phi_{k,\infty}\|^2$ using Jensen's inequality:
\begin{align}
	\|w^o(\eta)\!-\!\phi_{k,\infty}\|^2 &\leq \sum_{\ell=1}^N a_{1,\ell k} \|w^o(\eta)\!-\!w_{k,\infty}\|^2
\end{align}
Substituting into \eqref{eq:zeta_diff}, we get 
\begin{align}
	\E \|\zeta_{k,\infty} \!-\! \boldsymbol{\zeta}_{k,i}\|^2 &\leq (\gamma_k^2 + 2\mu^2 \alpha) \E \|\phi_{k,\infty}\!-\!\boldsymbol{\phi}_{k,i-1}\|^2 +\nonumber\\
								&\!\!\!\!\!\!\!\!\!\!\!\!\!\!\!\!\!\!\!\!\!\!\!\!\!\!\!\!\!\!\!\!\!\!\!\!\!\!\!   2\mu^2 \alpha \sum_{\ell=1}^N\! a_{1, \ell k} \|w^o(\eta)\!-\!w_{k,\infty}\|^2 \!\!+\!\! \mu^2 (2\alpha \|w^o(\eta)\|^2\!+\!\sigma_v^2) \label{eq:final_relation_zeta}
\end{align}
Now, combining \eqref{eq:final_relation_w}, \eqref{eq:final_relation_psi}, \eqref{eq:final_relation_zeta}, and \eqref{eq:final_relation_phi}, we obtain the following recursion for $\E P[\w_i\!-\!w_\infty]$:
\begin{align}
	\E P[\w_i\!-\!w_\infty] &\preceq A_2^\T \Gamma A_1^\T \E P[\w_{i-1}\!-\!w_\infty] + \mu^2 b \label{eq:w_i-w_infty}
\end{align}
where $\Gamma\in \mathbb{R}^{N\times N}$ is a diagonal matrix with elements $\gamma_k^2+2\mu^2\alpha$ along the diagonal, and $b\in\mathbb{R}^N$ is defined as
\begin{align}
	b &\triangleq 2 \alpha A_2^\T A_1^\T \E P[\mathds{1}_N\!\otimes\! w^o(\eta)\! -\! w_\infty] \!+\! (\sigma_v^2 \!+\! 2\alpha \|w^o(\eta)\|^2) \mathds{1}_N \label{eq:bv}
\end{align}
We prove in the next section that $b \in O((\mu\eta)^2) + O(1)$ --- see \eqref{eq:bias_final}. Iterating recursion \eqref{eq:w_i-w_infty} we obtain
\begin{align}
\E P[\w_i\!-\!w_\infty] \!&\preceq\! (A_2^\T \Gamma A_1^\T)^i \E P[\w_{0}\!-\!w_\infty] \!+\! \mu^2 \sum_{j=0}^{i-1} (A_2^\T \Gamma A_1^\T)^j b \label{eq:iterated_recur}
\end{align}
Observe that the matrix $A_2^\T \Gamma A_1^\T$ can be guaranteed to be stable for small step-sizes. To see this, we upper-bound the spectral radius by the matrix norm $\|B\|_\infty$, which is the maximum-absolute-row-sum:
\begin{align*}
	\rho(A_2^\T \Gamma A_1^\T) &\leq \|A_2^\T \Gamma A_1^\T\|_\infty \leq \|\Gamma\|_\infty= \max_{1\leq k \leq N} \{\gamma_k^2 + 2\mu^2 \alpha\}
\end{align*}
since $A_1$ and $A_2$ are left-stochastic matrices. We conclude then that the matrix $A_2^\T \Gamma A_1^\T$ is stable when
\begin{align}
	0 \!<\! \mu \!<\! \min_{1\leq k\leq N} \left\{\frac{2\lambda_{k,\min}}{\lambda_{k,\min}^2+2\alpha},\frac{2\lambda_{k,\max}}{\lambda_{k,\max}^2+2\alpha}\right\}
\end{align}
In this case, we have that that, using \eqref{eq:iterated_recur}, the triangle inequality, and the submultiplicative property of induced norms, 
\begin{align}
	&\left\|\limsup_{i\rightarrow\infty} \E P[\w_i\!-\!w_\infty]\right\|_\infty \leq \mu^2\|b\|_\infty \left\|\sum_{j=0}^\infty (A_2^\T \Gamma A_1^\T)^j\right\|_\infty \nonumber\\
	&\leq \mu^2\|b\|_\infty \sum_{j=0}^\infty \|\Gamma\|_\infty^j \leq \mu^2 \|b\|_\infty \sum_{j=0}^\infty (\gamma^2 + 2\mu^2\alpha)^j \nonumber\\
	&= \frac{\mu^2 \cdot \|b\|_\infty}{1-\gamma^2-2\mu^2 \alpha}\label{eq:limsup_MSP1}
\end{align}
where $\gamma$ was defined in \eqref{eq:gamma}. Combining \eqref{eq:gamma}--\eqref{eq:gamma_k}, we have that $\gamma^2$ can be obtained as
\begin{align}
	\gamma^2 \!&=\!\!\! \max_{1\leq k\leq N} \!\!\left\{\!1\!-\!2\mu \lambda_{k,\min} \!\!+\!\! \mu^2 \lambda_{k,\min}^2, \!1\!-\!2\mu \lambda_{k,\max} \!\!+\!\! \mu^2 \lambda_{k,\max}^2\right\}\nonumber\\
			\!&=\! 1\!-\! \mu \!\!\min_{1\leq k\leq N} \!\!\left\{\!2 \lambda_{k,\min} \!\!-\! \mu \lambda_{k,\min}^2, 2 \lambda_{k,\max} \!-\! \mu \lambda_{k,\max}^2\!\right\} \label{eq:gamma_alt_form}
\end{align}
Substituting \eqref{eq:gamma_alt_form} into \eqref{eq:limsup_MSP1}, we obtain
\begin{align}
	&\left\|\limsup_{i\rightarrow\infty} \E P[\w_i\!-\!w_\infty]\right\|_\infty \label{eq:limsup_MSP}\\
	&\leq \frac{\mu \cdot \|b\|_\infty}{\min_{k} \left\{2 \lambda_{k,\min} \!-\! \mu \lambda_{k,\min}^2, 2 \lambda_{k,\max} \!-\! \mu \lambda_{k,\max}^2\right\}-2\mu \alpha} \nonumber
\end{align}
Therefore, using the fact that $\|b\|=O((\mu\eta)^2)+O(1)$, as shown further ahead in \eqref{eq:bias_final}, we conclude that $\limsup_{i\rightarrow\infty} \E P[\w_i\!-\!w_\infty] \in O(\mu)$.

\subsection{Bias Analysis at Small Step-Sizes}
We now examine the dependence of $\|\mathds{1}_N\otimes w^o(\eta)-w_\infty\|^2$ on $\mu$; this term appears in expression \eqref{eq:bv} for $b$. First, we will derive an expression for $\tilde{w}_\infty \triangleq \mathds{1}_N \otimes w^o(\eta) - w_\infty$. For the remainder of this appendix, we will write $w^o\triangleq w^o(\eta)$ in order to simplify the notation. Our arguments will still apply for any $\eta > 0$. Recall that $w_\infty$ is the fixed point for the diffusion strategy in the absence of gradient noise. Therefore, let $i\rightarrow\infty$ in \eqref{eq:diff_C1}--\eqref{eq:diff_C2} in the absence of noise, and introduce the bias vectors $\tilde{w}_{k,\infty} = w^o - w_{k,\infty}$, $\tilde{\phi}_{k,\infty} = w^o - \phi_{k,\infty}$, $\tilde{\zeta}_{k,\infty} = w^o - \zeta_{k,\infty}$, and $\tilde{\psi}_{k,\infty} = w^o - \psi_{k,\infty}$. Subtracting $\phi_{k,\infty}$, $\zeta_{k,\infty}$, $\psi_{k,\infty}$, and $w_{k,\infty}$ from $w^o$ yields,
\begin{subequations}
\begin{align}
\tilde{\phi}_{k,\infty} &= \sum_{\ell = 1}^N a_{1,\ell k} \tilde{w}_{\ell, \infty} \label{eq:phi_err_recur}\\
\tilde{\zeta}_{k,\infty} &= \tilde{\phi}_{k,\infty} + \mu \nabla_w J_k(\phi_{k,\infty}) \label{eq:zeta_k_tilde_infty}\\
\tilde{\psi}_{k,\infty} &= \tilde{\zeta}_{k,\infty} + \mu\eta\nabla_w p_k(\zeta_{k,\infty}) \label{eq:psi_k_tilde_infty}\\
\tilde{w}_{k,\infty} &= \sum_{\ell=1}^N a_{2,\ell k} \tilde{\psi}_{\ell,\infty} \label{eq:w_err_recur}
\end{align}
\end{subequations}
Using the mean-value-theorem \cite[p.~6]{Polyak}, we can write
\begin{align}
	\nabla_w J_k(\phi_{k,\infty}) 
	&= \nabla_w J_k(w^o) - H_{k,\infty} \cdot \tilde{\phi}_{k,\infty}
\end{align}
where
\begin{align}
	H_{k,\infty} \triangleq \int_0^1 \nabla_w^2 J_k(w^o - t \tilde{\phi}_{k,\infty})dt
\end{align}
Therefore, \eqref{eq:zeta_k_tilde_infty} becomes
\begin{align}
	\tilde{\zeta}_{k,\infty} &= \left[I_M - \mu H_{k,\infty}\right]\cdot \tilde{\phi}_{k,\infty} + \mu \nabla_w J_k(w^o)
\end{align}
Similarly, we can obtain for \eqref{eq:psi_k_tilde_infty} that
\begin{align}
	\tilde{\psi}_{k,\infty} 
	&= \left[I_M-\mu\eta Z_{k,\infty}\right]\cdot \tilde{\zeta}_{k,\infty} + \mu\eta\nabla_w p_k(w^o)
\end{align}
where 
\begin{align}
	Z_{k,\infty} &\triangleq \int_0^1 \nabla_w^2 p_k(w^o-t\tilde{\zeta}_{k,\infty})) dt
\end{align}
To proceed, we introduce the extended quantities:
\begin{align*}
	\mathcal{A}_1 &\triangleq A_1 \otimes I_M, \quad\quad\quad\quad\quad\quad\quad\!\! \mathcal{A}_2 \triangleq A_2 \otimes I_M\\
	\mathcal{H}_\infty &\triangleq \textrm{diag}\{H_{1,\infty},...,H_{N,\infty}\}, \quad\!\!
	\mathcal{Z}_\infty \triangleq \textrm{diag}\!\left\{\! Z_{1,\infty},...,Z_{N,\infty}\!\right\} \\
	g^o &\triangleq \textrm{col}\{\nabla_w J_1(w^o),\ldots,\nabla_w J_N(w^o)\}\\
	f^o &\triangleq \textrm{col}\left\{\nabla_w p_1(w^o),\ldots,\nabla_w p_N(w^o)\right\}
\end{align*}
as well as the network error vector $\tilde{w}_\infty = \textrm{col}\{\tilde{w}_{1,\infty},\ldots,\tilde{w}_{N,\infty}\}$. Using these block variables, recursions \eqref{eq:phi_err_recur}-\eqref{eq:w_err_recur} lead to the following expression for $\tilde{w}_{\infty}$:
\begin{equation}
	\!\!\begin{aligned}
	\tilde{w}_\infty \!&=\! \left[\!I_{MN}\!-\!\mathcal{A}_2^\T \!\!\left(I_{MN}\!-\!\mu\eta \mathcal{Z}_\infty\!\right)\!\left(I_{MN}\!-\!\mu \mathcal{H}_\infty\right)\!\mathcal{A}_1^\T\right]^{-1} \!\!\times\\
	&\quad\,\left[\mu \mathcal{A}_2^\T \left(I_{MN}-\mu\eta \mathcal{Z}_\infty \right)g^o + \mu\eta \mathcal{A}_2^\T f^o\right]
	\end{aligned}
	\label{eq:w_infty_bias}
\end{equation}
when the inverse exists.
The matrix is invertible when $\!\mathcal{A}_2^\T \!\!\left(I_{MN}\!-\!\mu\eta \mathcal{Z}_\infty\!\right)\!\left(I_{MN}\!-\!\mu \mathcal{H}_\infty\right)\!\mathcal{A}_1^\T$ is stable. Since the spectral radius of a matrix is upper-bounded by any of its induced norms, we have that
\begin{align}
	\rho(\mathcal{A}_2^\T \!\!&\left(I_{MN}-\mu\eta \mathcal{Z}_\infty\!\right)\left(I_{MN}-\mu \mathcal{H}_\infty\right)\mathcal{A}_1^\T) \nonumber\\
	&\leq \|I_{MN}-\mu\eta \mathcal{Z}_\infty\|_{b,\infty} \cdot \|I_{MN}-\mu \mathcal{H}_\infty\|_{b,\infty}
\end{align}
where $\|\cdot\|_{b,\infty}$ denotes the block-maximum norm \cite{Jianshu_common_wo,book_chapter}. Now, it is sufficient to show that $\|I_{MN}-\mu \mathcal{H}_\infty\|_{b,\infty} < 1$ and $\|I_{MN}-\mu\eta \mathcal{Z}_\infty\|_{b,\infty} \leq 1$. For the former, observe that
\begin{align}
	\|I_{MN} - \mu \mathcal{H}_\infty\|_{b,\infty} = \max_{1\leq k\leq N}\ \{\|I_{MN} - \mu H_{k,\infty}\|_2\}
\end{align}
and due to Assumption \ref{ass:Hessian_cost},
\begin{align}
	(1 \!-\! \mu \lambda_{k,\max}) I_M \!\leq\! I_M \!-\! \mu H_{k,\infty} \leq (1\!-\!\mu \lambda_{k,\min}) I_M
\end{align}
We conclude that $\|I_{MN} - \mu H_{k,\infty}\|_2 \leq \gamma_ k$ where $\gamma_k$ is defined in \eqref{eq:gamma_k} and that $\|I_{MN} - \mu \mathcal{H}_\infty\|_{b,\infty} =  \max_{1\leq k\leq N}\ \gamma_k$. Similarly, using Assumption \ref{ass:Hessian_penalty}, it can be verified that
\begin{align}
\|I_{MN} - \mu \mathcal{Z}_\infty\|_{b,\infty} =  \max_{1\leq k\leq N}\!\max\{1,|1-\mu\eta \lambda_{k,\max}^p|\}
\end{align}
Finally, observe that $\gamma_k < 1$ is satisfied for all $1\leq k \leq N$ when $\mu$ is chosen according \eqref{eq:bound_stability2}. Also, $\max\{1,|1-\mu\eta \lambda_{k,\max}^p|\} \leq 1$ is satisfied for $1\leq k \leq N$ when $\mu\eta$ is chosen according to \eqref{eq:condition_nu_1}.

Comparing \eqref{eq:w_infty_bias} with expression (85) in \cite{Jianshu_diff_wo}, it is clear now how the current set-up is different and leads to additional challenges in the analysis. Observe that expression \eqref{eq:w_infty_bias} contains the additional terms $\mathcal{Z}_{\infty}$ and $f^o$, which are due to the penalty functions. If these terms are set to zero, then \eqref{eq:w_infty_bias} simplifies to expression (85) in \cite{Jianshu_diff_wo}. Moreover, we rewrite \eqref{eq:w_infty_bias} as:
\begin{align}
	\tilde{w}_\infty &=  \left[I_{MN}-\mathcal{A}_2^\T \mathcal{A}_1^\T + \mu \mathcal{A}_2^\T\mathcal{K}_\infty\mathcal{A}_1^\T\right]^{-1} \times\nonumber\\
	&\quad\,\left[ \mu \mathcal{A}_2^\T \left(g^o + \eta f^o - \mu\eta \mathcal{Z}_\infty g^o \right)\right] \label{eq:w_infty_bias2}
\end{align}
where we introduced the matrix:
\begin{align}
\mathcal{K}_\infty &\triangleq  \eta \mathcal{Z}_\infty + \mathcal{H}_\infty - \mu\eta \mathcal{Z}_\infty \mathcal{H}_\infty 
\end{align}
Observe that if $\mathcal{Z}_{\infty}=0$, then $\mathcal{K}_{\infty}$ would be a symmetric matrix, which is the case studied in \cite{Jianshu_diff_wo} in the context of unconstrained optimization. Here, the penalty functions introduce the additional factor $\mathcal{Z}_{\infty}$, in addition to $f^o$ in \eqref{eq:w_infty_bias2}. 

Our goal now is to show that
\begin{align}
\lim_{\mu \rightarrow 0} \frac{\|\mathds{1}\otimes w^o-w_\infty\|}{\mu} = C
\end{align}
for some constant $C$ that may be dependent on $\eta$ (the approximation parameter), but not $\mu$ (the algorithm parameter). To begin with, we introduce the Jordan canonical decomposition of the matrix $A_2^\T A_1^\T = T^{-\T} D T^\T$ so that
\begin{equation}
	\mathcal{A}_2^\T \mathcal{A}_1^\T \!=\! A_2^\T A_1^\T \!\otimes\! I_M \!=\! (T^{-\T} \!\otimes\! I_M) (D \!\otimes\! I_M) (T^\T \!\otimes\! I_M)
\end{equation}
Then, we may re-write \eqref{eq:w_infty_bias2} as
\begin{align}
	\tilde{w}_\infty &= (T^{-\T} \otimes I_M) \left[I_{MN}-D\otimes I_M + \mu E\right]^{-1}\times\nonumber\\
	&\quad\,(T^{\T} \otimes I_M) \left[ \mu \mathcal{A}_2^\T \left(g^o + \eta f^o - \mu\eta \mathcal{Z}_\infty g^o\right)\right] \label{eq:w_infty_bias3}\\
	E &\triangleq (T^\T \otimes I_M) \mathcal{A}_2^\T\mathcal{K}_\infty\mathcal{A}_1^\T (T^{-\T} \otimes I_M) \label{eq:matrix_E}
\end{align}
By Assumption \ref{ass:primitive} we know that $A_2^\T A_1^\T$ is a doubly stochastic and primitive matrix. It follows from the Perron-Frobenius theorem \cite[pp.~730--731]{Papoulis} that $A_2^\T A_1^\T$ has a single eigenvalue at one with all other eigenvalues strictly inside the unit circle. Therefore, we may partition $D$ and $T$ as follows:
\begin{equation}
	D \!=\! \textrm{diag}\{1,D_0\}, \!\!\quad\!\! T^\T \!=\! \textrm{col}\left\{\mathds{1}^\T , T_R\right\},\!\!\quad\!\! T^{-\T} \!=\! [\mathds{1}, T_L]
	\label{eq:decomp_A_1A_2}
\end{equation}
where $D_0$ has a block Jordan structure satisfying $\rho(D_0) < 1$. Substituting \eqref{eq:decomp_A_1A_2} into \eqref{eq:matrix_E}, we can partition $E$ into blocks:
\begin{align}
E_{11} &\triangleq \left(\mathds{1}^\T \otimes I_M\right) \mathcal{A}_2^\T\mathcal{K}_\infty\mathcal{A}_1^\T (\mathds{1}\otimes I_M) \label{eq:E11}\\
E_{12} &\triangleq \left(\mathds{1}^\T \otimes I_M\right)\mathcal{A}_2^\T\mathcal{K}_\infty\mathcal{A}_1^\T (T_L \otimes I_M)\\
E_{21} &\triangleq (T_R \otimes I_M) \mathcal{A}_2^\T\mathcal{K}_\infty\mathcal{A}_1^\T (\mathds{1} \otimes I_M)\\
E_{22} &\triangleq (T_R \otimes I_M) \mathcal{A}_2^\T\mathcal{K}_\infty\mathcal{A}_1^\T (T_L \otimes I_M)
\end{align}
where $E_{ij}$ indicates the $(i,j)$-th block. Substituting into \eqref{eq:w_infty_bias3}:
\begin{align}
&\tilde{w}_\infty = (T^{-\T} \!\otimes I_M) \left[\begin{array}{cc}
\mu E_{11} & \mu E_{12} \\ 
\mu E_{21} & I-D_0\otimes I_M + \mu E_{22}
\end{array} \right]^{-1} \times\nonumber\\
& \left[\!\!\!\begin{array}{c}
 (\mu \mathds{1}^\T \!\!\otimes\!\! I_M) \mathcal{A}_2^\T \left(g^o \!+\! \eta f^o\right) \!-\! \mu^2 \eta (\mathds{1}^\T \!\!\otimes\!\! I_M) \mathcal{A}_2^\T \mathcal{Z}_\infty g^o  \\ 
\mu\cdot (T_R \otimes I_M) \mathcal{A}_2^\T \left(g^o + \eta f^o - \mu\eta \mathcal{Z}_\infty g^o \right)
\end{array} \!\!\!\right] \label{eq:w_infty_bias4}
\end{align}
Furthermore, recalling that $w^o$ is the solution of the minimization problem \eqref{eq:global_cost_approx}, we see that it is the root of
\begin{align}
	(\mathds{1}^\T \otimes I_M) (g^o + \eta f^o) = 0
\end{align}
Using the fact that the matrix $A_2$ is doubly stochastic, expression \eqref{eq:w_infty_bias4} simplifies to
\begin{align}
	\tilde{w}_\infty 
&= \mu\cdot (T^{-\T} \otimes I_M) \left[\!\!\begin{array}{cc}
\mu E_{11} & \mu E_{12} \\ 
\mu E_{21} & I-D_0\otimes I_M + \mu E_{22}
\end{array}\!\! \right]^{-1} \!\!\!\times \nonumber\\
	&\quad\, \left[\begin{array}{c}
- \mu\eta (\mathds{1}^\T \otimes I_M) \mathcal{A}_2^\T \mathcal{Z}_\infty g^o  \\ 
(T_R \otimes I_M) \mathcal{A}_2^\T \left(g^o + \eta f^o - \mu\eta \mathcal{Z}_\infty g^o \right)
\end{array} \right] \label{eq:tilde_w_infty}
\end{align}
Let us denote
\begin{align}
	\!\!\!G \triangleq \!\left[\!\!\begin{array}{cc}
	G_{11} & G_{12} \\ 
	G_{21} & G_{22}
	\end{array}\!\!\right] \!\!=\!\! \left[\!\!\begin{array}{cc}
\mu E_{11} & \mu E_{12} \\ 
\mu E_{21} & I-D_0\otimes I_M + \mu E_{22}
\end{array} \!\!\right]^{-1}\!\!\! \label{eq:matrixG}
\end{align}
Observe that $G$ is invertible since $G^{-1}$ is similar to $I_{MN}-\mathcal{A}_2^\T \mathcal{A}_1^\T + \mu \mathcal{A}_2^\T\mathcal{K}_\infty\mathcal{A}_1^\T$ from \eqref{eq:w_infty_bias2}, which we have already shown to be invertible when \eqref{eq:condition_nu_1} and \eqref{eq:bound_stability2} are satisfied. Then, from \eqref{eq:tilde_w_infty}, $\tilde{w}_{\infty}$ is given by:
\begin{align}
	\tilde{w}_\infty &= \mu (T^{-\T} \otimes I_M) G \left[\begin{array}{c}
	-\mu\eta \cdot p_1\\ 
	p_2
	\end{array} \right]\\
	p_1 &\triangleq (\mathds{1}^\T \otimes I_M) \mathcal{A}_2^\T \mathcal{Z}_\infty g^o  \\ 
	p_2 &\triangleq (T_R \otimes I_M) \mathcal{A}_2^\T \left(g^o + \eta f^o - \eta\mu \mathcal{Z}_\infty g^o \right)
\end{align}
Applying the block inversion formula \cite[p.~48]{laub} to \eqref{eq:matrixG},
\begin{align}
	&\lim_{\mu \rightarrow 0} \frac{\|\tilde{w}_\infty\|}{\mu} = \lim_{\mu \rightarrow 0}  \Bigg\|(T^{-\T} \otimes I_M)\times\nonumber\\
	&\left[\!\!\!\!\begin{array}{cc}
-\eta E_{11}^{-1} p_1 -\mu \eta E_{11}^{-1} E_{12} G_{22} E_{21} E_{11}^{-1} p_1  -E_{11}^{-1} E_{12} G_{22} p_2 \\ 
\mu \eta G_{22} E_{21} E_{11}^{-1}p_1 + G_{22} p_2
\end{array}\!\!\!\!\right]\!\!\Bigg\| \nonumber
\end{align}
But the right-hand-side is constant since the only matrices with dependence on $\mu$ are $G_{22}$ and $p_2$, which satisfy:
\begin{align}
	G_{22,\infty} &\triangleq \lim_{\mu\rightarrow 0} G_{22} = \left(I_{MN} - D_0\otimes I_M\right)^{-1}\\
	p_{2,\infty} &\triangleq \lim_{\mu\rightarrow 0} p_2 = (T_R \otimes I_M) \mathcal{A}_2^\T \left(g^o + \eta f^o\right)
\end{align}
so we have
\begin{align}
	\lim_{\mu \rightarrow 0} \frac{\|\tilde{w}_\infty\|}{\mu} \!&=\! 
\left\|(T^{-\T} \otimes I_M)\! \left[\!\!\begin{array}{cc}
I & -E_{11}^{-1}E_{12} \\ 
0 & I
\end{array} \!\!\right]\! \left[\!\!\begin{array}{c}
-\eta E_{11}^{-1} p_1 \\ 
G_{22,\infty}p_{2,\infty}
\end{array} \!\!\right]\right\| \nonumber\\
&= O(\eta) \nonumber 
\end{align}
We conclude that 
\begin{align}
	\|\tilde{w}_\infty\|^2 \in O((\mu \eta)^2) 	\label{eq:bias_final}
\end{align}
if $\eta = \mu^{-\theta}$ with $0 < \theta < 1$. Therefore, the bias $P[\mathds{1}_N \otimes w^o(\eta) - w_\infty]$ diminishes with $\mu^2$. Since the bias appears in \eqref{eq:limsup_MSP} through the vector $b$ defined in \eqref{eq:bv}, we conclude that $b \rightarrow (\sigma_v^2 \!+\! 2\alpha \|w^o(\eta)\|^2) \mathds{1}_N$ at a rate of $O((\mu\eta)^2)$ and therefore \eqref{eq:limsup_MSP} is $O(\mu)$. The second term of \eqref{eq:triangle_inequality} is, as we just established, $O((\mu \eta)^2)$ We conclude, therefore that 
\begin{align}
\limsup_{i\rightarrow\infty} \E\|\mathds{1}_N \otimes w^o(\eta) - \w_i\| \leq O(\mu) + O((\mu \eta)^2)
\end{align}
which is \eqref{eq:result_SS_small_mu}.  

\ifCLASSOPTIONcaptionsoff
  \newpage
\fi



%
\bibliographystyle{IEEEbib}
\bibliography{refs}


\end{document}